\documentclass[final,3p]{elsarticle}
\usepackage{epstopdf}
\usepackage{amssymb,amsmath}
\usepackage{amsthm}
\newtheorem{thm}{Theorem}
\newtheorem{lem}[thm]{Lemma}
\newtheorem{example}{Example}

\newtheorem{definition}{Definition}
\newtheorem{remark}{Remark}
\newtheorem{theorem}{Theorem}

\usepackage{booktabs}
\usepackage{multirow}
\usepackage{subfigure}
\usepackage{algorithm}
\usepackage{algorithmic}
\usepackage{lineno}
\usepackage[colorlinks=true,breaklinks=true,pdftex]{hyperref}

\journal{}

\begin{document}

\begin{frontmatter}

\title{Preconditioned geometric iterative methods for cubic B-spline interpolation curves}

\author[first]{Chengzhi Liu}
\cortext[cor]{lizhang02121@126.com}
\author[first]{Yue Qiu}
\author[second]{Li Zhang\corref{cor}}
\address[first]{School of Mathematics and Finance, Hunan University of Humanities, Science and Technology, Loudi, P.R. China}
\address[second]{Data Recovery Key Laboratory of Sichuan Province, College of Mathematics and Information Science, Neijiang Normal University, Neijiang 641100, P. R. China}

\begin{abstract}
The geometric iterative method (GIM) is widely used in data interpolation/fitting, but its slow convergence affects the computational efficiency. Recently, much work was done to guarantee the acceleration of GIM in the literature. In this work, we aim to further accelerate the rate of convergence by introducing a preconditioning technique. After constructing the preconditioner, we preprocess the progressive iterative approximation (PIA) and its variants, called the preconditioned GIMs. We show that the proposed preconditioned GIMs converge and the extra computation cost brought by the preconditioning technique is negligible. Several numerical experiments are given to demonstrate that our preconditioner can accelerate the convergence rate of PIA and its variants.
\end{abstract}


\end{frontmatter}


\section{Introduction}\label{sec1}
Data fitting arises in a variety of scientific and engineering applications, including geometric modeling, image processing, data mining, and others.
The rapid development of science and technology makes it easy to access massive data, which also brings a great challenge to data fitting techniques.
In recent years, the rise of the geometric iterative method (GIM) provided a stable and highly efficient way for data fitting.
The GIM has the advantages of simple iterative format, stable convergence, and clear geometric meaning and thus intrigued many scholars for years (\cite{Lin201722, Lin2013}).

The GIM originated from the profit and loss property for data fitting by using uniform cubic B-spline curves (\cite{Qi1975}). This property was extended to non-uniform cubic B-spline curves and surfaces over two decades later (\cite{Lin2004}). In 2005, Lin et al. pointed out that the normalized and totally positive basis has the property of profit and loss and called it progressive iterative approximation (PIA) (\cite{Lin2005}). Thereafter, more systematic research on the PIA emerged and various forms of PIA were put forward consequently, see \cite{Lin2010,Chen2011,Lin2011,Deng2014, Lin2017,LiuM20181,Hamza2020}. In 2007, another branch of the GIM, namely geometric interpolation (GI), was proposed by Takashi Maekawa et al \cite{Maekawa2007}. The principle of the GI is similar to that of the PIA. Since then, several similar approaches were put forward, e.g., \cite{Gofuku2009, Lin20101,Xiong2012,Yuki2012}. In \cite{Lin201722}, Lin et al. summarized the PIAs, the GIs, and their applications. Due to their clear geometric meaning, such a class of iterative methods is collectively referred to as the GIM.

Although the GIMs have the advantage of stable convergence, they are likely to suffer from slow convergence for problems that arise from typical applications such as curves and surfaces reconstruction (\cite{Liu20211,Liu2021,Liu2020,Liu20201}). To remedy this, much work was done to accelerate the GIMs' convergence rate, and several acceleration methods and alternatives have been proposed in the literature, see \cite{Liu20211,Liu2021,Liu20201,Liu2020,Lu2010,Carnicer2010,Carnicer2011,Zhang2018,Ebrahimi2019,Liu2015,Wang2018,Hu2020,Liu2022,Shou2022} and so on.

It is well known that iterative methods combined with preconditioning techniques work surprisingly well when the preconditioners are selected appropriately. Consequently, preconditioning techniques for the acceleration of the PIA were deservedly proposed \cite{Liu2021,Liu20201,Liu2020}. Despite the fact that preconditioning techniques could accelerate the rate of convergence significantly, they always require extra costs to compute preconditioning operations. Very often, one has to take into account the convergence rate and computational complexity when constructing the preconditioner. To reduce the computational complexity, the inexact versions of preconditioned PIA were proposed (\cite{Liu2021,Liu20201}). In order to reduce the computational cost caused by the preconditioning operation, we in this paper study the preconditioning technique for the GIMs, in which the convergence is accelerated and the extra cost is small enough to ignore.

The rest of this paper is organized as follows: After reviewing the PIA for cubic B-spline curves in Section \ref{sec2}, we fetch out some variants of PIA. In Section \ref{sec3}, we exploit the preconditioned geometric iterative methods for cubic B-spline curves and analyze their convergence. Section \ref{sec4} gives some numerical examples to illustrate the acceleration of the preconditioning technique. We end with some conclusions in the last section.

\section{Related work}\label{sec2}
\subsection{PIA and its variants}
Consider interpolating a given set of organized data points $\{\boldsymbol{p}_{i}\}_{i=1}^{n}$ in $\mathbb{R}^3\ \textrm{or}\ \mathbb{R}^3$, whose parameters are $t_i(i=1,2,\ldots,n)$, respectively. We construct a knots vector with multiple-knots
$\boldsymbol{T}=\{t_i\}_{i=-2}^{n+3}$ subject to $t_{-2}=t_{-1}=t_{0}=t_{1}<t_{2}<\ldots<t_{n}=t_{n+1}=t_{n+2}=t_{n+3}$, then we define a cubic B-spline basis $\{N_{i}^{3}(t)\}_{i=-2}^{n-1}$ on the knots vector $\boldsymbol{T}$.

In the PIA, we begin with an initial interpolation B-spline curve
$
\boldsymbol{C}^{(0)}(t)=\sum_{i=-2}^{n-1}\boldsymbol{p}_{i+2}^{(0)}N_{i}^{3}(t),\ t\in[t_1,t_n],
$
where $\boldsymbol{p}_{0}^{(0)}=\boldsymbol{p}_{1},\ \boldsymbol{p}_{i}^{(0)}=\boldsymbol{p}_{i},i= 1,\ldots, n, \ \boldsymbol{p}_{n+1}^{(0)}=\boldsymbol{p}_{n}.$
Then we compute the difference vector $\boldsymbol{\delta}_{i}^{(0)}=\boldsymbol{p}_{i}- \boldsymbol{C}^{(0)}(t_i), i= 1,\ldots, n$ and update the control points according to
\begin{displaymath}
\left\{ \begin{array}{ll}
\boldsymbol{p}_{i}^{(1)}=\boldsymbol{p}_{i}^{(0)}+ \boldsymbol{\delta}_{i}^{(0)},\  \text{if}\ i\in \{1,\ldots, n\};\\
\boldsymbol{p}_{i}^{(1)}=\boldsymbol{p}_{i}^{(0)}, \  \qquad \quad \text{if}\ i\in \{0, n+1\}.
 \end{array} \right.
\end{displaymath}
Consequently, we can update the interpolation B-spline curve
$\boldsymbol{C}^{(1)}(t)=\sum_{i=-2}^{n-1}\boldsymbol{p}_{i+2}^{(1)}N_{i}^{3}(t),\ t\in[t_1,t_n].$

Assume that we have obtained the approximate interpolation B-spline curve $\boldsymbol{C}^{(k)}(t)$ after $k$ iterations. Then we can compute $\boldsymbol{\delta}_{i}^{(k)}=\boldsymbol{p}_{i}-\boldsymbol{C}^{(k)}(t_i), i= 1,\ldots, n$ and generate the $(k+1)$-th approximate interpolation B-spline curve
$\boldsymbol{C}^{(k+1)}(t)=\sum_{i=-2}^{n-1}\boldsymbol{p}_{i+2}^{(k+1)}N_{i}^{3}(t),\ t\in[t_1,t_n],$
where
\begin{equation}\label{equ1:point11}
\left\{ \begin{array}{ll}
\boldsymbol{p}_{i}^{(k+1)}=\boldsymbol{p}_{i}^{(k)}+ \boldsymbol{\delta}_{i}^{(k)},\  \text{if}\ i\in \{1,\ldots, n\};\\
\boldsymbol{p}_{i}^{(k+1)}=\boldsymbol{p}_{i}^{(k)}, \ \qquad \quad \text{if}\ i\in \{0, n+1\}.
 \end{array} \right.
\end{equation}
Thus we obtain a sequence of cubic B-spline curves $\{\boldsymbol{C}^{(k)}(t)\}_{k=0}^{\infty}$ that approximately interpolate $\{\boldsymbol{p}_{i}\}_{i=1}^{n}$. The method to generate the curves sequence is known as the PIA. We note in \cite{Lin2004} that the limit of curves sequence $\{\boldsymbol{C}^{(k)}(t)\}_{k=0}^{\infty}$ interpolates $\{\boldsymbol{p}_{i}\}_{i=1}^{n}$, i.e.,
$  \lim\limits_{k\rightarrow\infty}\boldsymbol{C}^{(k)}(t_i)=\boldsymbol{p}_{i},\ i=1,2,\ldots,n.$

Let $\boldsymbol{p}=[\boldsymbol{p}_1,\ldots,\boldsymbol{p}_n]^{T}$, $\boldsymbol{p}^{(k)}=[\boldsymbol{p}_1^{(k)},\ldots,\boldsymbol{p}_n^{(k)}]^{T}$. Then the PIA format for updating the control points of the curves sequence can be arranged in the matrix form
\begin{equation} \label{equ:matrixPIA}
\boldsymbol{p}^{(k+1)} = (I-B)\boldsymbol{p}^{(k)} + \boldsymbol{p},
\end{equation}
where $I$ is the identity matrix of order $n$, and $B$ is the so-called collocation matrix, i.e.,
\begin{equation}\label{equ:collm}
  \small
 B =  \left[ {\begin{array}{*{20}c}
   1 & 0 &  &  & \\
   N_{-1}^{3}(t_2) & N_{0}^{3}(t_2) & N_{1}^{3}(t_2) &   \\
    &  \ddots  &  \ddots  &  \ddots  & {}  \\
    &  & N_{n-4}^{3}(t_{n-1}) & N_{n-3}^{3}(t_{n-1}) & N_{n-2}^{3}(t_{n-1}) \\
    &  &  & 0  & 1  \\
\end{array}} \right].
\end{equation}

Note in \cite{Liu2020} that the iteration \eqref{equ:matrixPIA} is mathematically equivalent to the Richardson method for solving the collocation equations
\begin{equation} \label{equ:collequ}
B\boldsymbol{x} = \boldsymbol{p}.
\end{equation}

\subsection{Variants of PIA}
To improve the convergence of PIA, some acceleration methods were proposed. Note that different adjustment strategies for updating the control points result in variants of PIA, some of them are listed as follows (the adjustment vectors are boxed):
\begin{enumerate}
          \item[(1)] \textbf{WPIA} (\cite{Lu2010}). At each iteration of WPIA, the difference vector $\boldsymbol{\delta}_{i}^{(k)}$ is multiplied by a weight $\omega\ (0<\omega<2)$ when updating the control points, i.e.,
          \begin{displaymath}
          \begin{split}
                \boldsymbol{p}_{i}^{(k+1)}&=\boldsymbol{p}_{i}^{(k)}+ \omega\boldsymbol{\delta}_{i}^{(k)}\\
                      &=\boldsymbol{p}_{i}^{(k)}+\boxed{\omega(\boldsymbol{p}_{i}- \boldsymbol{p}_{i-1}^{(k)}N_{i-3}^{3}(t_i)-\boldsymbol{p}_{i}^{(k)}N_{i-2}^{3}(t_i)-\boldsymbol{p}_{i+1}^{(k)}N_{i-1}^{3}(t_i))}, i = 1, 2, \ldots, n.
          \end{split}
          \end{displaymath}

          \item[(2)] \textbf{Jacobi--PIA} (\cite{Liu2015}). At each iteration of WPIA, the difference vector $\boldsymbol{\delta}_{i}^{(k)}$ is multiplied by $\frac{1}{N_{i-2}^{3}(t_i)}$ when updating the control points, i.e.,
          \begin{displaymath}
          \begin{split}
                \boldsymbol{p}_{i}^{(k+1)}&=\boldsymbol{p}_{i}^{(k)}+ \frac{1}{N_{i-2}^{3}(t_i)}\boldsymbol{\delta}_{i}^{(k)}\\
                      &=\boldsymbol{p}_{i}^{(k)}+\boxed{\frac{\boldsymbol{p}_{i}- \boldsymbol{p}_{i-1}^{(k)}N_{i-3}^{3}(t_i)-\boldsymbol{p}_{i}^{(k)}N_{i-2}^{3}(t_i)-\boldsymbol{p}_{i+1}^{(k)}N_{i-1}^{3}(t_i)}{N_{i-2}^{3}(t_i)}}, i = 1, 2, \ldots, n.
          \end{split}
          \end{displaymath}
          \item[(3)] \textbf{GS--PIA} (\cite{Wang2018}). At each iteration of GS--PIA, we make full use of the calculated control points at the current step when generating the $(k+1)$-th B-spline curve, i.e.,
          $$\boldsymbol{p}_{i}^{(k+1)}=\boldsymbol{p}_{i}^{(k)} + \boxed{\frac{\boldsymbol{p}_{i} - \boldsymbol{p}_{i-1}^{(k+1)}{N_{i-3}^{3}(t_i)}  - \boldsymbol{p}_{i}^{(k)} N_{i-2}^{3}(t_i) - \boldsymbol{p}_{i+1}^{(k)}N_{i-1}^{3}(t_i)}{N_{i-2}^{3}(t_i)}}, i = 1, 2, \ldots, n.$$
          \item[(4)] \textbf{SOR--PIA} (\cite{Shou2022}). Based on the GS--PIA, we in SOR--PIA multiply the adjustment vector of GS--PIA by a weight $\omega\ (0<\omega<2)$
          $$\boldsymbol{p}_{i}^{(k+1)}=\boldsymbol{p}_{i}^{(k)} + \boxed{\omega \frac{\boldsymbol{p}_{i}-\boldsymbol{p}_{i-1}^{(k+1)}N_{i-3}^{3}(t_i)-\boldsymbol{p}_{i}^{(k)} N_{i-2}^{3}(t_i)-\boldsymbol{p}_{i+1}^{(k+1)}N_{i-1}^{3}(t_i)}{N_{i-2}^{3}(t_i)}}, i = 1, 2, \ldots, n.$$
\end{enumerate}

Given a matrix $A=(a_{ij})\in \mathbb{R}^{n\times n}$, we denote by $D_{A}$, $-L_{A}$, and $-U_{A}$ the diagonal, the strict lower part, and the strict upper part of $A$, respectively. By direct deduction, the matrix forms of the WPIA, the Jacobi--PIA, the GS--PIA and the SOR--PIA are
\begin{itemize}
          \item WPIA: $\boldsymbol{p}^{(k+1)} = (I-\omega B)\boldsymbol{p}^{(k)} + \omega\boldsymbol{p}$,

          \item Jacobi--PIA: $\boldsymbol{p}^{(k+1)} = D_{A}^{-1}(L_{B}+U_{B})\boldsymbol{p}^{(k)} + D_{B}^{-1}\boldsymbol{p}$,

          \item GS--PIA: $\boldsymbol{p}^{(k+1)} = D_{B}^{-1}(L_{B}\boldsymbol{p}^{(k+1)} + U_{B}\boldsymbol{p}^{(k)}+ \boldsymbol{p})$ or equivalently $\boldsymbol{p}^{(k+1)} = (D_{B}-L_{B})^{-1}U_{B}\boldsymbol{p}^{(k)} + (D_{B}-L_{B})^{-1}\boldsymbol{p}$,

          \item SOR--PIA: $\boldsymbol{p}^{(k+1)} =
          (1-\omega)\boldsymbol{p}^{(k)}+\omega( D_{B}^{-1}(L_{B}\boldsymbol{p}^{(k+1)} + U_{B}\boldsymbol{p}^{(k)}+ \boldsymbol{p}))$ or equivalently $\boldsymbol{p}^{(k+1)} = (D_{B}-\omega L_{B})^{-1}((1-\omega)D_{B}+\omega U_{B})\boldsymbol{p}^{(k)} + \omega(D_{B}-\omega L_{B})^{-1}\boldsymbol{p}$.
\end{itemize}

\begin{remark}\label{rem1}
   We remark here that all the aforementioned iterations can be seen as the basic iterative methods for solving \eqref{equ:collequ}. Let
   $B=M-N$ be a splitting of $B$, and $M$ be an invertible matrix. Then the splitting iteration for solving \eqref{equ:collequ} can be written as
   \begin{equation}\label{equ:splittingitertion}
     \boldsymbol{x}^{(k+1)}=M^{-1}N\boldsymbol{x}^{(k)}+M^{-1}\boldsymbol{p},\quad k=0,1,\ldots,
   \end{equation}
   where $\boldsymbol{x}^{(0)}$ is the initial guess, and $M^{-1}N$ is called iteration matrix. It is well known that the iteration \eqref{equ:splittingitertion} is convergent of the spectral radius of the iteration matrix is less than 1, i.e., $\rho(M^{-1}N)<1$.

   For the $D_{B}$, $L_{B}$, and $U_{B}$ defined above, we have
    \begin{itemize}
          \item If $M = I$, the splitting iteration \eqref{equ:splittingitertion} is the Richardson iteration for solving the collocation system \eqref{equ:collequ} and is equivalent to the PIA.
          \item If $M = \frac{1}{\omega} I$, the splitting iteration \eqref{equ:splittingitertion} is the modified Richardson iteration for solving the collocation system \eqref{equ:collequ} and is equivalent to the WPIA.
          \item If $M = D_{B}$, the splitting iteration \eqref{equ:splittingitertion} is the Jacobi iteration for solving the collocation system \eqref{equ:collequ} and is equivalent to the Jacobi--PIA.
          \item If $M = D_{B} - L_{B}$, the splitting iteration \eqref{equ:splittingitertion} is the Gauss--Seidel iteration for solving the collocation system \eqref{equ:collequ} and is equivalent to the GS--PIA.
          \item If $M = \frac{1}{\omega}(D_{B} - \omega L_{B})$, the splitting iteration \eqref{equ:splittingitertion} is the SOR iteration for solving the collocation system \eqref{equ:collequ} and is equivalent to the SOR--PIA.
\end{itemize}
\end{remark}

\section{Preconditioning geometric iterative method}\label{sec3}
Although the PIA and its variants mentioned above are convergent, researchers are more likely to exploit iteration formats with faster convergence because they are more efficient in data interpolation. It is well known that preconditioning is a key technique for improving the efficiency and robustness of iterative methods. A suitable and problem-dependent choice of preconditioner can often achieve unexpected results. In this section, we will discuss the preconditioned versions of geometric iterative methods.
\subsection{Construction of the preconditioner}
The preconditioner for the GIM is defined as
\begin{equation}\label{equ:precon}
Q = I + S,
\end{equation}
where $S$ is a super-diagonal matrix whose entries are the minus of the super-diagonal entries of $B$, i.e.,
\begin{displaymath}
  \small
 S =  \left[ {\begin{array}{*{20}c}
   0 & 0 &  &  & \\
    & 0 & -N_{1}^{3}(t_2) &   \\
    &    &  \ddots  &  \ddots  & {}  \\
    &  &  & 0 & -N_{n-2}^{3}(t_{n-1}) \\
    &  &  &  & 0  \\
\end{array}} \right].
\end{displaymath}
\subsection{Preconditioning techniques}
We preconditioning the system \eqref{equ:collequ} with the preconditioner $Q$ and obtain the preconditioned system
\begin{equation}\label{equ:precollocation}
  QB\boldsymbol{x}=Q\boldsymbol{p}
\end{equation}
\begin{remark}
It should be noted that most of the preconditioning techniques require calculating the inverse of the preconditioner, and additional computational costs brought by inversion should be taken into account. In the preconditioned system \eqref{equ:precollocation}, there is no need for us to compute the inverse of the preconditioner. And the additional computation bought by our preconditioning technique only involves with the multiplication of a bi-diagonal matrix and a tri-diagonal matrix. In this way, the sparsity can be fully utilized to reduce the amount of computation.
\end{remark}

From \eqref{equ:collm} and \eqref{equ:precon}, we have
\begin{equation} \label{equ:sorPIA}
  \begin{split}
    QB &= (I+S)(D_{B}-L_{B}-U_{B})\\
       &= D_{B}-L_{B}-U_{B} +SD_{B}-SL_{B}-SU_{B}\\
       &= (D_{B}-SL_{B})-L_{B}-(U_{B} - SD_{B} + SU_{B}),
  \end{split}
  \end{equation}
then ${D}_{QB}=D_{B}-SL_{B},\ -{L}_{QB} = - L_{B}$, and $-{U}_{QB} = -U_{B} + SD_{B} - SU_{B}$ are the diagonal, the strict lower part, and the strict upper part of $QB$, respectively. By employing the Richardson, the modified Richardson, the Jacobi, the Gauss-Seidel, and the SOR iterations to the preconditioned system \eqref{equ:precollocation}, we can obtain the preconditioned GIMs given as follows:
\begin{enumerate}
  \item[(1)] \textbf{Preconditioned PIA:}
  \begin{equation} \label{equ:PPIA}
  \begin{split}
    \boldsymbol{p}^{(k+1)} &= (I-QB)\boldsymbol{p}^{(k)} + Q\boldsymbol{p},\\
                           &= \boldsymbol{p}^{(k)} + Q(\boldsymbol{p}-B\boldsymbol{p}^{(k)})
  \end{split}
  \end{equation}
  \item[(2)] \textbf{Preconditioned WPIA:}
  \begin{equation} \label{equ:PWPIA}
  \begin{split}
    \boldsymbol{p}^{(k+1)} &= (I-\omega QB)\boldsymbol{p}^{(k)} + \omega Q\boldsymbol{p},\\
                           &= \boldsymbol{p}^{(k)} + \omega Q(\boldsymbol{p}- B\boldsymbol{p}^{(k)})
  \end{split}
  \end{equation}
  \item[(3)] \textbf{Preconditioned Jacobi--PIA:}
  \begin{equation} \label{equ:PjPIA}
  \begin{split}
    \boldsymbol{p}^{(k+1)} &= - {D}_{QB}^{-1}({L}_{QB}+{U}_{QB})\boldsymbol{p}^{(k)} + {D}_{QB}^{-1} Q\boldsymbol{p},\\
                           &= \boldsymbol{p}^{(k)} + {D}_{QB}^{-1}Q(\boldsymbol{p}- B\boldsymbol{p}^{(k)}).
    \end{split}
  \end{equation}
  \item[(4)] \textbf{Preconditioned GS--PIA:}
  \begin{equation} \label{equ:gsPIA}
    \boldsymbol{p}^{(k+1)} = ({D}_{QB}-{L}_{QB})^{-1}{U}_{QB}\boldsymbol{p}^{(k)} + ({D}_{QB}-{L}_{QB})^{-1} Q\boldsymbol{p}.
  \end{equation}
  \item[(5)] \textbf{Preconditioned SOR--PIA:}
  \begin{equation} \label{equ:sorPIA}
    \boldsymbol{p}^{(k+1)} = ({D}_{QB}-\omega {L}_{QB})^{-1}\left((1-\omega){D}_{QB}+\omega {U}_{QB}\right)\boldsymbol{p}^{(k)} + \omega({D}_{QB}-\omega{L}_{QB})^{-1} Q\boldsymbol{p}.
  \end{equation}
\end{enumerate}

The relaxation factors $\omega$ in \eqref{equ:PWPIA} and \eqref{equ:sorPIA} are introduced to accelerate the rate of convergence. When $\omega=1$, the preconditioned WPIA will be reduced to the preconditioned PIA, and the preconditioned SOR--PIA will be reduced to the preconditioned GS--PIA. Therefore, we need to find the optimal choices of $\omega$, which will be given in the following subsection.
\subsection{Convergence analysis}
Before analyzing the convergence, we introduce some definitions and conclusions.
\begin{definition}[\cite{Minc1988,Horn1985,Saad2009,Fallat2011}]\label{deftotal}
  A matrix $A$ is said to be nonnegative if all its entries are nonnegative, denoted by $A\geq 0$. A matrix is said to be stochastic if it is nonnegative and all its row sums are $1$. A matrix is said to be totally positive if all its minors are nonnegative. A matrix $A=(a_{ij})\in \mathbb{R}^{n\times n}$ is said to be sign-regular if the sign of the $(i,j)$-th entry is $(-1)^{i+j}$. A matrix $A=(a_{ij})\in \mathbb{R}^{n\times n}$ is said to be a Z-matrix if $a_{ij}\leq 0, i\neq j$. A matrix $A=(a_{ij})\in \mathbb{R}^{n\times n}$ is said to be an M-matrix if $A$ is a Z-matrix and $A^{-1}\geq 0$. Let $A=M-N$ be the splitting of $A=(a_{ij})\in \mathbb{R}^{n\times n}$. Then the pair of matrices $M$, $N$ is a regular splitting of $A$ if $M$ is nonsingular and $M^{-1}$ and $N$ are nonnegative.
\end{definition}
\begin{definition}\label{defcom}
  Let $A=(a_{ij})\in \mathbb{R}^{n\times n}$. Then the matrix $\langle A \rangle=(\langle a_{ij}\rangle)$ is called the comparison matrix of $A$ if
  \begin{displaymath}
\langle a_{ij}\rangle =\left\{ \begin{array}{ll}
|a_{ij}|,&  i=j;\\
-|a_{ij}|,& i\neq j.
 \end{array} \right.
\end{displaymath}
\end{definition}
\begin{lem}[\cite{Minc1988}]\label{lemnonnegative}
  Let $A,C\in \mathbb{R}^{n\times n}$ be nonnegative matrices. Then, $A+C\geq 0$ and $AC\geq 0$.
\end{lem}
\begin{lem}[\cite{Fallat2011}]\label{lemsigregular}
  Let $A\in \mathbb{R}^{n\times n}$ be a be sign-regular matrix. Then, there exists a sign matrix $J = \emph{diag}\big(1,-1,1,\cdots,$ $ (-1)^{(n+1)}\big)$ such that $JAJ$ is nonnegative, and $JAJ=|A|$.
\end{lem}
\begin{lem}[\cite{Fallat2011}]\label{lemtotal}
  Let $A\in \mathbb{R}^{n\times n}$ be a nonsingular matrix. Then, $A$ is totally positive if and only if $A^{-1}$ is sign-regular and $JA^{-1}J=|A^{-1}|\geq 0$, where $J = \emph{diag}\left(1,-1,1,\cdots, (-1)^{(n+1)}\right)$.
\end{lem}
\begin{lem}[\cite{Minc1988}]\label{lemmmatrix}
  A nonsingular matrix $A$ is an M-matrix if and only if $A^{-1}$ is nonnegative.
\end{lem}
\begin{lem}[\cite{Saad2009}]\label{lem1}
  Let $A=M-N$ be a regular splitting of $A=(a_{ij})\in \mathbb{R}^{n\times n}$. Then, $\rho(M^{-1}N)<1$ if and only if $A$ is nonsingular and $A^{-1}$ is nonnegative.
\end{lem}
\begin{lem}[\cite{Horn1985}]\label{lem6}
  If $A$ is a matrix with $\rho(A)<1$, then $I-A$ is nonsingular, and
  $$(I-A)^{-1}= 1+A+A^2+\cdots.$$
\end{lem}

\begin{theorem}\label{thelem}
  Let $B$ be the collocation matrix defined as in \eqref{equ:collm}, and $Q$ be the preconditioner defined as in \eqref{equ:precon}. Then,
  \begin{enumerate}
    \item[\emph{(i)}] ${D}_{QB}$ is a diagonal matrix with positive diagonals.
    \item[\emph{(ii)}] $\langle QB\rangle = J QB J$, where $J = \emph{diag}\left(1,-1,1,\cdots, (-1)^{(n+1)}\right)$.
    \item[\emph{(iii)}] The preconditioner $Q$ is invertible and $Q^{-1}$ is a sign-regular matrix.
    \item[\emph{(iv)}] $\left(\left\langle QB\right\rangle\right)^{-1}$ is nonnegative and $\left\langle QB\right\rangle$ is a M-matrix.
  \end{enumerate}
\end{theorem}
\begin{proof}
  Denote by $(QB)_{ij}$ the $(i,j)$-th entry of the matrix $QB$, it follows from \eqref{equ:collm} and \eqref{equ:precon} that
\begin{equation}\label{equ:collm1}
 \left\{ \begin{array}{ll}
(QB)_{11}=(QB)_{nn}=1, (QB)_{n-1,n-1}=N_{n-3}^{3}(t_{n-1}),\\
(QB)_{i,i-1}=N_{i-3}^{3}(t_i), & i = 2,\ldots,n-1;\\
(QB)_{ii}=N_{i-2}^{3}(t_i)-N_{i-1}^{3}(t_i)N_{i-2}^{3}(t_{i+1}), & i = 2,\ldots,n-2;\\
(QB)_{i,i+1}=N_{i-1}^{3}(t_i)-N_{i-1}^{3}(t_{i+1})N_{i-1}^{3}(t_{i}), & i = 2,\ldots,n-2;\\
(QB)_{i,i+2}=-N_{i-1}^{3}(t_{i})N_{i}^{3}(t_{i+1}), & i = 2,\ldots,n-2;\\
(QB)_{ij} = 0,& \textrm{others}.
 \end{array} \right.
\end{equation}

  To prove (i), note in \cite{Lin2004} that the collocation matrix $B$ defined as in \eqref{equ:collm} is a stochastic and totally positive matrix.
  It follows from \eqref{equ:collm} that $0< N_{i-3}^{3}(t_i), N_{i-2}^{3}(t_i), N_{i-1}^{3}(t_i)<1$, and $N_{i-3}^{3}(t_i)+ N_{i-2}^{3}(t_i)+N_{i-1}^{3}(t_i)=1$ for $i = 2,\ldots, n-1$. From the definition of totally positive, all the minors of $B$, of order $2$, are nonnegative, hence
\begin{displaymath}
  \left|
    \begin{array}{cc}
      N_{i-2}^{3}(t_i) & N_{i-1}^{3}(t_i) \\
      N_{i-2}^{3}(t_{i+1}) & N_{i-1}^{3}(t_{i+1}) \\
    \end{array}
  \right|=
  N_{i-2}^{3}(t_i)N_{i-1}^{3}(t_{i+1})-N_{i-1}^{3}(t_i)N_{i-2}^{3}(t_{i+1})\geq 0.
\end{displaymath}
Then, the diagonal entries
\begin{displaymath}
 (QB)_{ii}=\left\{ \begin{array}{ll}
1>0,&i=1,n;\\
N_{i-2}^{3}(t_i)-N_{i-1}^{3}(t_i)N_{i-2}^{3}(t_{i+1})>0,&i = 2,\ldots,n-2;\\
N_{n-3}^{3}(t_{n-1})>0, & i = n-1.\\
 \end{array} \right.
\end{displaymath}
Thus ${D}_{QB}$ is a diagonal matrix with positive diagonals.

 To prove (ii) we observe from \eqref{equ:collm1} that for $|i-j|=1$, $(QB)_{ij}>0$, and for $j-i=2$, $(QB)_{ij}<0$. We can obtain the conclusion by direct calculation.

To prove (iii), according to \eqref{equ:precon}, the determinant of $Q$ is $1$, hence $Q$ is invertible. Let $X =(x_{ij})$ be the inverse of $Q$, and let $\boldsymbol{x}_{j}=({x}_{1j},{x}_{2j},\ldots, {x}_{nj})^{T}$ be the $j$-th column of $X$. Then $Q\boldsymbol{x}_{j}= \boldsymbol{e}_{j}$, where $\boldsymbol{e}_{j}$ is an unit vector with the $j$-th entry equals to 1. Since $Q$ is a banded upper triangular matrix, the $\boldsymbol{x}_{j}$ can be obtained by solving $Q\boldsymbol{x}_{j}= \boldsymbol{e}_{j}$, that is,
\begin{displaymath}
 {x}_{ij} =\left\{ \begin{array}{ll}
 0,& i>j;\\
 1,& i=j;\\
  (-1)^{i-j}\prod\limits_{l=i}^{j-1}N^3_{l-1} (t_l),   & i<j,
 \end{array} \right.
\end{displaymath}
where $N^3_{0} (t_1)=0$. Hence $X$ is a is a sign-regular upper triangular matrix, thus the conclusion (iii) holds.

To prove (iv), according to Lemma \ref{lemmmatrix}, we need to show that $\left(\langle QB\rangle\right)^{-1}$ is nonnegative. Notice that $J^{-1} = J$. From (ii), we have $\left(\langle QB\rangle\right)^{-1} = (J QB J)^{-1} = J B^{-1}Q^{-1} J = J B^{-1}JJQ^{-1} J$. Since the collocation matrix $B$ defined as in \eqref{equ:collm} is totally positive. It follows from Lemma \ref{lemtotal} that $JB^{-1}J\geq 0$. From (iii), since $Q^{-1}$ is a sign-regular matrix, $JQ^{-1}J$ is nonnegative. Therefore, it follows from Lemma \ref{lemnonnegative} that $\left(\langle QB\rangle\right)^{-1} = J B^{-1}JJQ^{-1} J$ is nonnegative. Thus the conclusion (iv) holds.
\end{proof}

\begin{theorem}\label{the:pia}
  The preconditioned PIA \eqref{equ:PPIA} is convergent.
\end{theorem}
\begin{proof}
  To prove the conclusion, we need to show the spectral radius of the preconditioned PIA is less than $1$, i.e., $\rho\left(I-QB\right)< 1$. Consider the Richardson splitting of the comparison matrix of $ QB$, that is, $\langle QB\rangle = M_{\langle QB\rangle}-N_{\langle QB\rangle},$
  where $M_{\langle QB\rangle} = I$, $N_{\langle QB\rangle} = I-\langle QB\rangle$.
  Clearly, both $M_{\langle QB\rangle}^{-1}$ and $N_{\langle QB\rangle}$ are nonnegative.
  Thus the Richardson splitting of $\langle QB\rangle$ is a regular one. Therefore, it follows from Theorem \ref{thelem}(iv) and Lemma \ref{lem1} that $\rho\left(M_{\langle QB\rangle}^{-1}N_{\langle QB\rangle}\right)=\rho\left(I-\langle QB\rangle\right)< 1$.
   One the other hand, the Richardson splitting of $QB$ is $ QB = M_{ QB}-N_{ QB},$ where $M_{ QB} = I$, $N_{ QB} = I- QB$. From Theorem \ref{thelem}(ii), since $I-\langle QB\rangle = J(I- QB)J $, we have that $I-\langle QB\rangle$ is similar to $I- QB$, and hence $\rho\left(I- QB\right)=\rho\left(I-\langle QB\rangle\right)< 1$. This completes the proof.
\end{proof}

\begin{theorem}\label{the:PWPIA}
  When $\omega =2\bigg/\left(\min\limits_{i=1}^{n}|\lambda_i(QB)|+\max\limits_{i=1}^{n}|\lambda_i(QB)|\right)$, the preconditioned WPIA \eqref{equ:PWPIA} converges and has the fastest convergence, where $\lambda_i(QB), i=1,2,\ldots,n$ are the eigenvalues of $QB$.
\end{theorem}
\begin{proof}
  We want to find the optimal $\omega$ such that
  \begin{displaymath}
  \begin{split}
    \omega_{\textrm{opt}} &= \arg\min_{\omega \in \mathbb{C}}\rho(I-\omega QB)\\
                        &= \arg\min_{\omega \in \mathbb{C}}\max_{i=1}^{n} |1-\omega \lambda_i(QB)|\\
                        &= \arg\min_{\omega \in \mathbb{C}}\max \left\{\left|1-\omega\min_{i=1}^{n}|\lambda_i(QB)|\right|,  \left|1-\omega\max_{i=1}^{n}|\lambda_i(QB)|\right|\right\},
    \end{split}
  \end{displaymath}
   Therefore, the optimal relaxation factor $\omega$ arrives when $1-\omega\min\limits_{i=1}^{n}|\lambda_i(QB)|=-1+\omega\max\limits_{i=1}^{n}|\lambda_i(QB)|$. It yields the conclusion.
\end{proof}

\begin{theorem}\label{the:jacobi}
  The preconditioned Jacobi--PIA \eqref{equ:PjPIA} is convergent.
\end{theorem}
\begin{proof}
  Consider the Jacobi splitting of the comparison matrix of $QB$, i.e., $\langle QB\rangle = M_{\langle QB\rangle}-N_{\langle QB\rangle},$
  where $M_{\langle QB\rangle} = {D}_{QB}=D_{B}-SL_{B}$, $N_{\langle QB\rangle} = L_{B} + U_{B} - SD_{B} + SU_{B}$.
  From Definition \ref{defcom} and Theorem \ref{thelem}(i), it is easy to verify that $M_{\langle QB\rangle}^{-1}$ and $N_{\langle QB\rangle}$ are nonnegative.
  Thus the splitting $\langle QB\rangle = M_{\langle QB\rangle}-N_{\langle QB\rangle}$ is a regular one. Therefore, it follows from Theorem \ref{thelem}(iv) and Lemma \ref{lem1} that $\rho\left(M_{\langle QB\rangle}^{-1}N_{\langle QB\rangle}\right)< 1$.
   One the other hand, the Jacobi splitting of the $QB$ is $ QB = M_{ QB}-N_{ QB},$ where $M_{ QB} = {D}_{QB}$, $N_{ QB} = -L_{B} - U_{B} + SD_{B} + SU_{B}$. Notice that for the sign matrix $J$, $J{D}_{QB}J={D}_{QB}$ and $J(-L_{B} - U_{B} + SD_{B} + SU_{B})J = L_{B} + U_{B} - SD_{B} + SU_{B}$, therefore, $JM_{\langle QB\rangle}^{-1}N_{\langle QB\rangle}J=M_{ QB}^{-1}N_{ QB}$. This means that $M_{\langle QB\rangle}^{-1}N_{\langle QB\rangle}$ is similar to $M_{ QB}^{-1}N_{ QB}$, and hence $\rho\left(M_{ QB}^{-1}N_{ QB}\right)=\rho\left(M_{\langle QB\rangle}^{-1}N_{\langle QB\rangle}\right)< 1$. This completes the proof.
\end{proof}
\begin{theorem}\label{the:gs}
  The preconditioned GS--PIA \eqref{equ:gsPIA} is convergent. 
\end{theorem}
\begin{proof}
  Consider the Gauss--Seidel splitting of the comparison matrix of $ QB$, i.e., $\langle QB\rangle = M_{\langle QB\rangle}-N_{\langle QB\rangle},$
  where $M_{\langle QB\rangle} = {D}_{\langle QB\rangle}-{L}_{\langle QB\rangle}=D_{B}-SL_{B}-L_{B}$, $N_{\langle QB\rangle} = U_{B} - SD_{B} + SU_{B}$. It is easy to verify that $N_{\langle QB\rangle}\geq 0$ and ${D}_{\langle QB\rangle}$ is a diagonal matrix with positive diagonals. Therefore, $\left(M_{\langle QB\rangle}\right)^{-1}=\left({D}_{\langle QB\rangle}-{L}_{\langle QB\rangle}\right)^{-1}=\left(I-{D}_{\langle QB\rangle}^{-1}{L}_{\langle QB\rangle}\right)^{-1}{D}_{\langle QB\rangle}^{-1}$.
  Since ${D}_{\langle QB\rangle}^{-1}{L}_{\langle QB\rangle}$ is a strict lower triangular matrix, all the eigenvalues of ${D}_{\langle QB\rangle}^{-1}{L}_{\langle QB\rangle}$ equal to $0$. Hence the spectral radius of ${D}_{\langle QB\rangle}^{-1}{L}_{\langle QB\rangle}$ is less than $1$, then it follows from Lemma \ref{lem6} that
  $
    ({D}_{\langle QB\rangle}-{L}_{\langle QB\rangle})^{-1}=\sum\limits_{k=0}^{\infty}\left({D}_{\langle QB\rangle}^{-1}{L}_{\langle QB\rangle}\right)^{-1}{D}_{\langle QB\rangle}^{-1}\geq 0.
  $
  Conbined with the condition $N_{\langle QB\rangle}\geq 0$, we can conclude that the splitting $\langle QB\rangle = (D_{QB}-L_{QB})-U_{QB}$ is a regular one. The rest of this proof is similar to that of Theorem \ref{the:jacobi}.
  \end{proof}


In \cite{Varga2000}, the optimal relaxation factor $\omega$ for the SOR iteration is discussed. Since the preconditioned SOR--PIA is equivalent to the SOR iteration for solving the linear system \eqref{equ:precollocation}, we in the following theorem present the optimal relaxation factor $\omega$ for the preconditioned SOR--PIA.
\begin{theorem}[\cite{Varga2000}]\label{the:sor}
  When $\omega=\frac{2}{1+\sqrt{1-\rho^2}}$, the preconditioned SOR--PIA \eqref{equ:PjPIA} converges and has the fastest convergence rate, where $\rho$ is the spectral radius of the iteration matrix of the preconditioned Jacobi--PIA.
\end{theorem}

\section{Numerical results}\label{sec4}
In this section some numerical experiments are conducted to access the effectiveness of the preconditioning technique for PIA and its variants. All these experiments were performed in Matlab.


In our tests, we employed the GIMs and their corresponding preconditioned GIMs to interpolate the data given in Examples \ref{eg1} - \ref{eg6}. We use
 \begin{equation}\label{error2}
 \varepsilon^{(k)} = \max_{1\leq i\leq n}\left\|{\boldsymbol{p}}_{i}-{\boldsymbol{C}}^{(k)}(t_{i})\right\|
\end{equation}
to measure the interpolation error of the $k$-th approximate interpolation curve ${\boldsymbol{C}}^{(k)}(t)$.
For simplicity, the preconditioned techniques are denoted by PPIA, PWPIA, PJacobi--PIA, PGS--PIA and PSOR--PIA, the number of iterations and the computing time (in seconds) are denoted by ``$k$'' and ``$T$'', respectively.

\begin{example}[Outline of a duck]\label{eg1}
Consider data interpolation of $40$ points:
\emph{(--0.2356, 0.3978), (--0.2044, 0.4178), (--0.1711, 0.4289), (--0.1467, 0.4733), (--0.1022, 0.4978), (--0.0533, 0.4933), (--0.0200, 0.4667), (0, 0.4444), (0.0089, 0.4111), (--0.0044, 0.3667), (--0.0333, 0.3311), (--0.0778, 0.2756), (--0.1067, 0.2400), (--0.1178, 0.2000), (--0.0889, 0.1778), (-- 0.0511, 0.2156), (0.0156, 0.2533), (0.0844, 0.2778), (0.1467, 0.2956), (0.2111, 0.2911), (0.2556, 0.2644), (0.2578, 0.2222), (0.2267, 0.1911), (0.2667, 0.1800), (0.2622, 0.1467), (0.2222, 0.1111), (0.2467, 0.0933), (0.2267, 0.0556), (0.1800, 0.0289), (0.0200, 0.0244), (--0.1311, 0.0267), (--0.1711, 0.0711), (--0.2133, 0.1356), (--0.2133, 0.2067), (--0.1822, 0.2622), (--0.1311, 0.3178), (--0.1000, 0.3733), (--0.1533, 0.3733), (--0.2178, 0.3689), (--0.2311, 0.3822), (--0.2356, 0.3978).}
\end{example}

\begin{example}[Butterfly curve]\label{eg2}
Consider data interpolation of $150$ points, which are sampled from
\begin{center}
$r=\left(\sin \theta+\sin (3.5 \theta)^{3}\right)/1000;  0 \leq \theta \leq 2 \pi.$
\end{center}
\end{example}

\begin{example}[Chrysanthemum curve]\label{eg3}
Consider data interpolation of $500$ points, which are sampled from
\begin{center}
$r=\left(5\left(1+\sin \frac{11 \theta}{5}\right)-4 \sin ^{4} \frac{17 \theta}{3} \sin ^{8}(2 \cos 3 \theta-28 \theta)\right) / 50;  0 \leq \theta \leq 21 \pi.$
\end{center}
\end{example}

\begin{example}[Spatial circular curve]\label{eg4}
Consider data interpolation of $300$ points, which are sampled from
\begin{center}
$\left\{\begin{array}{l}
x=(4+\sin 20 t) \cos t \\
y=(4+\sin 20 t) \sin t;  -7 \pi \leq t \leq 7 \pi. \\
z=\cos 20 t
\end{array}\right.$
\end{center}
\end{example}

\begin{example}[Three dimensional three leaf rose curve]\label{eg5}
Consider data interpolation of $200$ points, which are sampled from
\begin{center}
$\left\{\begin{array}{l}
x=\sin 3 t \cos t \\
y=\sin 3 t \sin t;  -2 \pi \leq t \leq 2 \pi. \\
z=t
\end{array}\right.$
\end{center}
\end{example}

\begin{example}[Spherical cardioid curve]\label{eg6}
Consider data interpolation of $n~(n=1000,2000)$ points, which are sampled from
\begin{center}
$\left\{\begin{array}{l}
x=2 \cos t-\cos 2 t \\
y=2 \sin t-\sin 2 t;  0<t<4 \pi. \\
z=\sqrt{8} \cos \frac{2}{t}
\end{array}\right.$
\end{center}
\end{example}
\begin{table}[H]
	\centering
	\caption{The spectral radii of iteration matrices of  GIMs and preconditioned GIMs in Examples \ref{eg1} - \ref{eg6}.}
	\label{tab1}
	\begin{tabular}{*{8}{c}}
		\toprule
         \multirow{2}*{Method} & \multirow{2}*{Example \ref{eg1}} & \multirow{2}*{Example \ref{eg2}} & \multirow{2}*{Example \ref{eg3}} & \multirow{2}*{Example \ref{eg4}} & \multirow{2}*{Example \ref{eg5}} & \multicolumn{2}{c}{Example \ref{eg6}}\\
		\cmidrule(r){7-8}
         &   &   &   &   &   &$n=1000$ & $n=2000$ \\
         \midrule
          PIA         & 0.6890 & 0.7252 & 0.9653 & 0.6666 & 0.6676 & 0.7049 & 0.7049\\	
          PPIA        & 0.6439 & 0.6791 & 0.9541 & 0.6070 & 0.6079 & 0.6588 & 0.6588\\		
          WPIA        & 0.5256 & 0.5689 & 0.9329 & 0.5000 & 0.5010 & 0.5443 & 0.5443\\	
          PWPIA       & 0.4748 & 0.5141 & 0.9122 & 0.4357 & 0.4367 & 0.4912 & 0.4912\\
          Jacobi-PIA  & 0.5065 & 0.5309 & 0.9329 & 0.5000 & 0.5000 & 0.5130 & 0.5130\\
          PJacobi-PIA & 0.3891 & 0.3928 & 0.8762 & 0.3847 & 0.3844 & 0.3956 & 0.3956\\
          GS-PIA      & 0.2566 & 0.2902 & 0.8703 & 0.3081 & 0.2974 & 0.3261 & 0.3299\\
          PGS-PIA     & 0.1204 & 0.1447 & 0.7676 & 0.1573 & 0.1504 & 0.1710 & 0.1739\\
          SOR-PIA     & 0.1053 & 0.2168 & 0.5405 & 0.2381 & 0.2241 & 0.2641 & 0.2674\\
          PSOR-PIA    & 0.0498 & 0.1021 & 0.3931 & 0.1156 & 0.1075 & 0.1290 & 0.1318\\
		\bottomrule
	\end{tabular}
\end{table}
\begin{figure}[H]
\centering
	 \subfigure[Example \ref{eg1}.]{
		\includegraphics[width=0.31\textwidth]{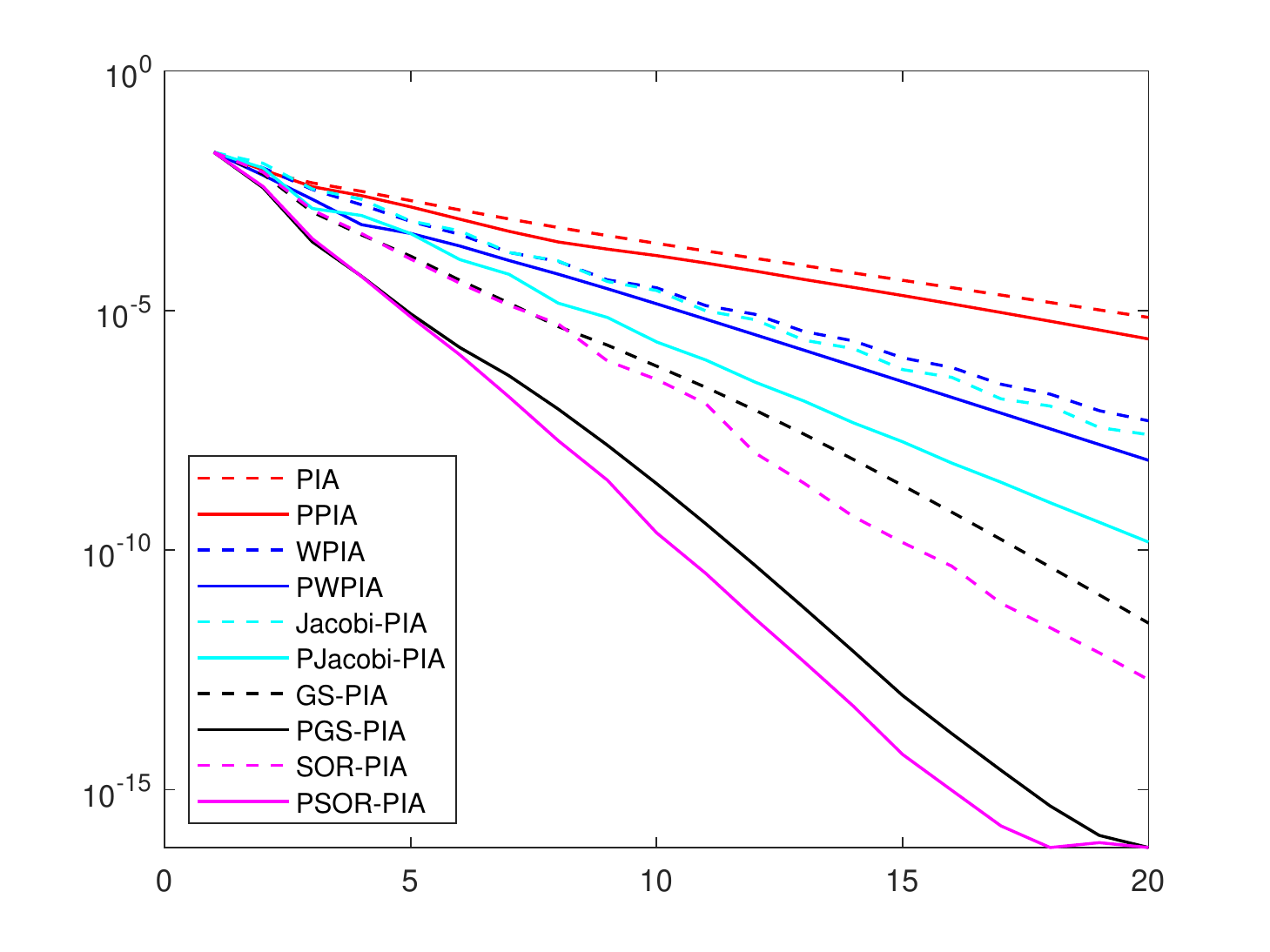}}
     \subfigure[Example \ref{eg2}.]{
		\includegraphics[width=0.31\textwidth]{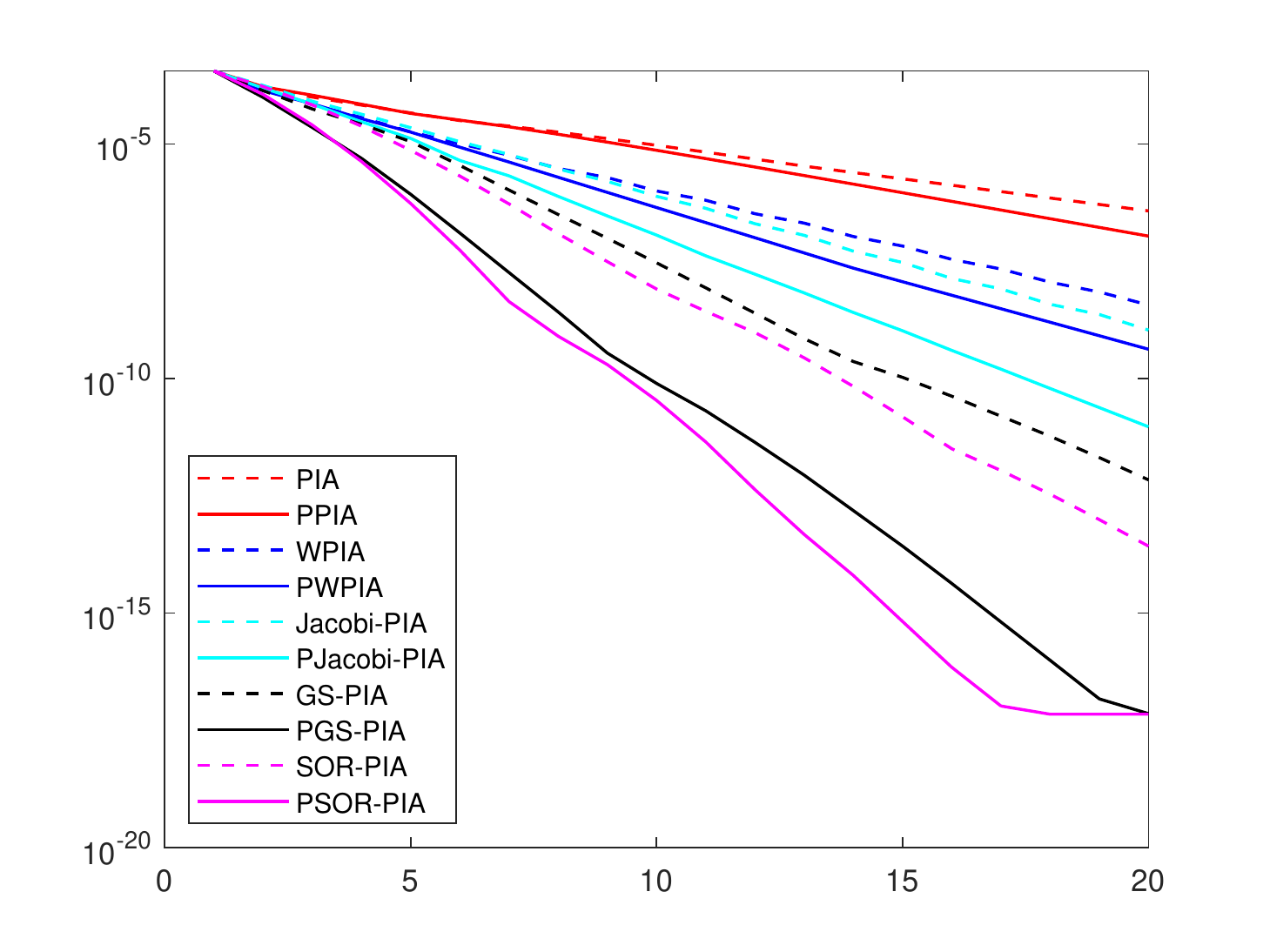}}
     \subfigure[Example \ref{eg3}.]{
		\includegraphics[width=0.31\textwidth]{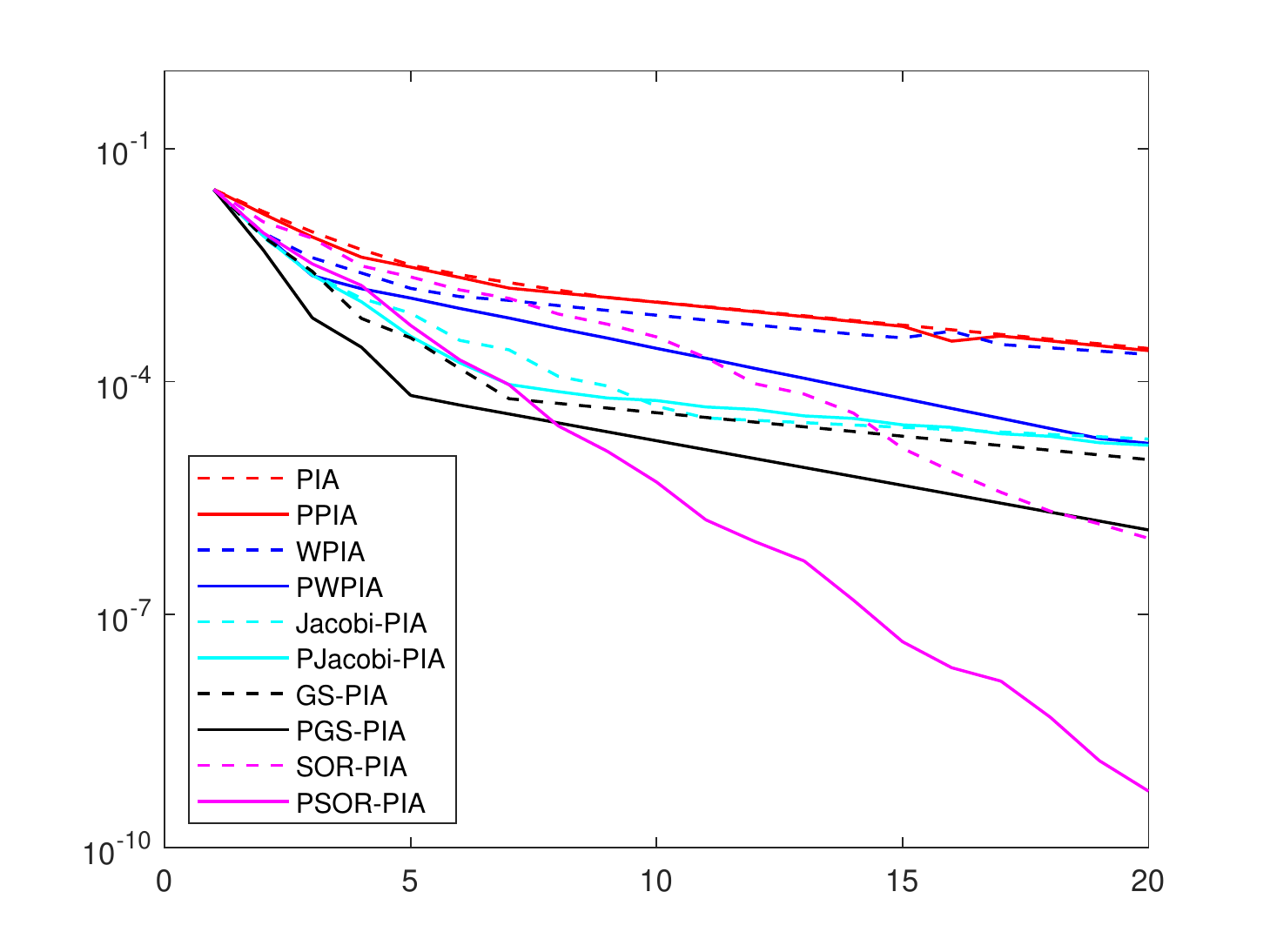}}
     \subfigure[Example \ref{eg4}.]{
		\includegraphics[width=0.31\textwidth]{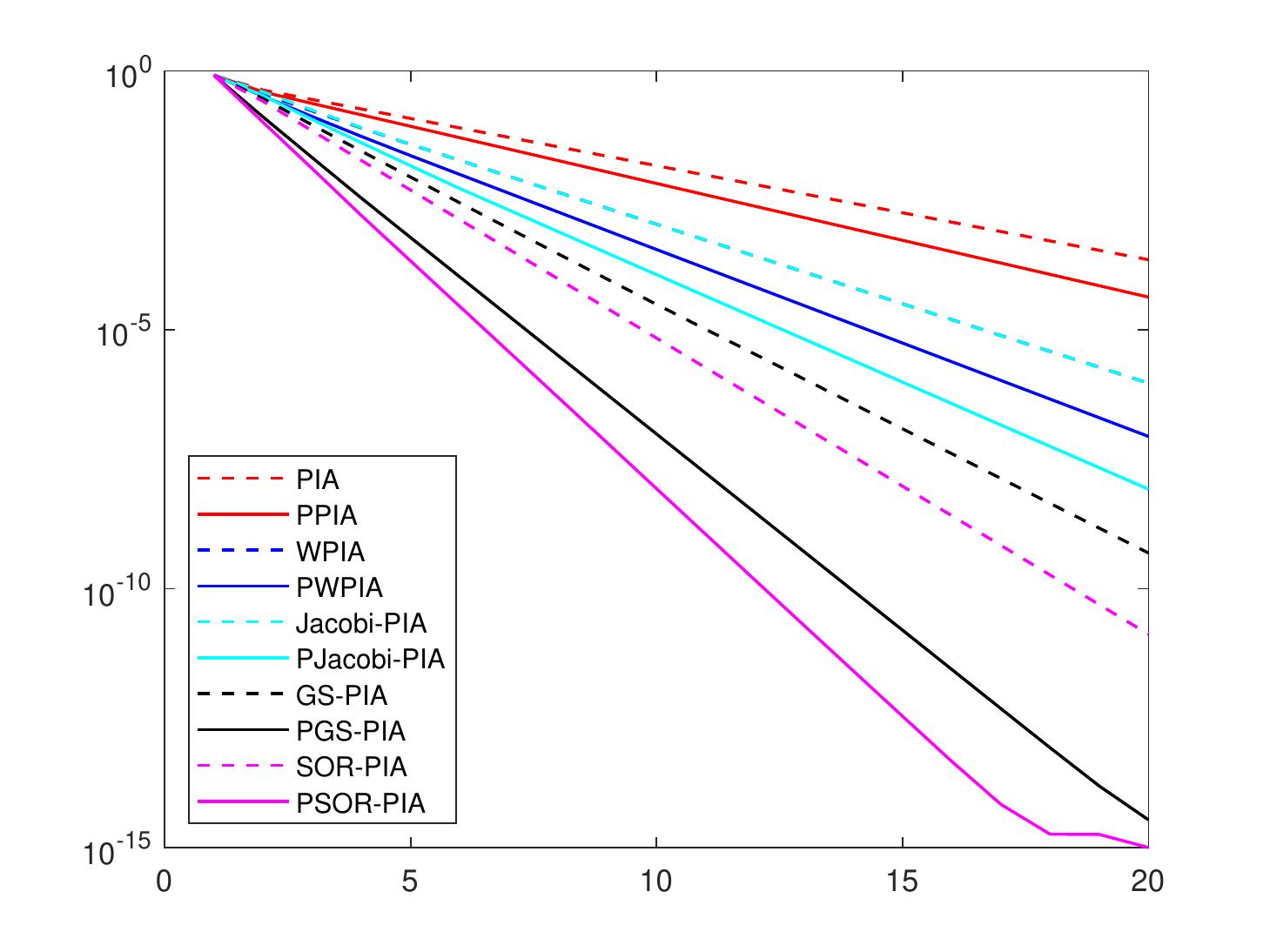}}
     \subfigure[Example \ref{eg5}.]{
		\includegraphics[width=0.31\textwidth]{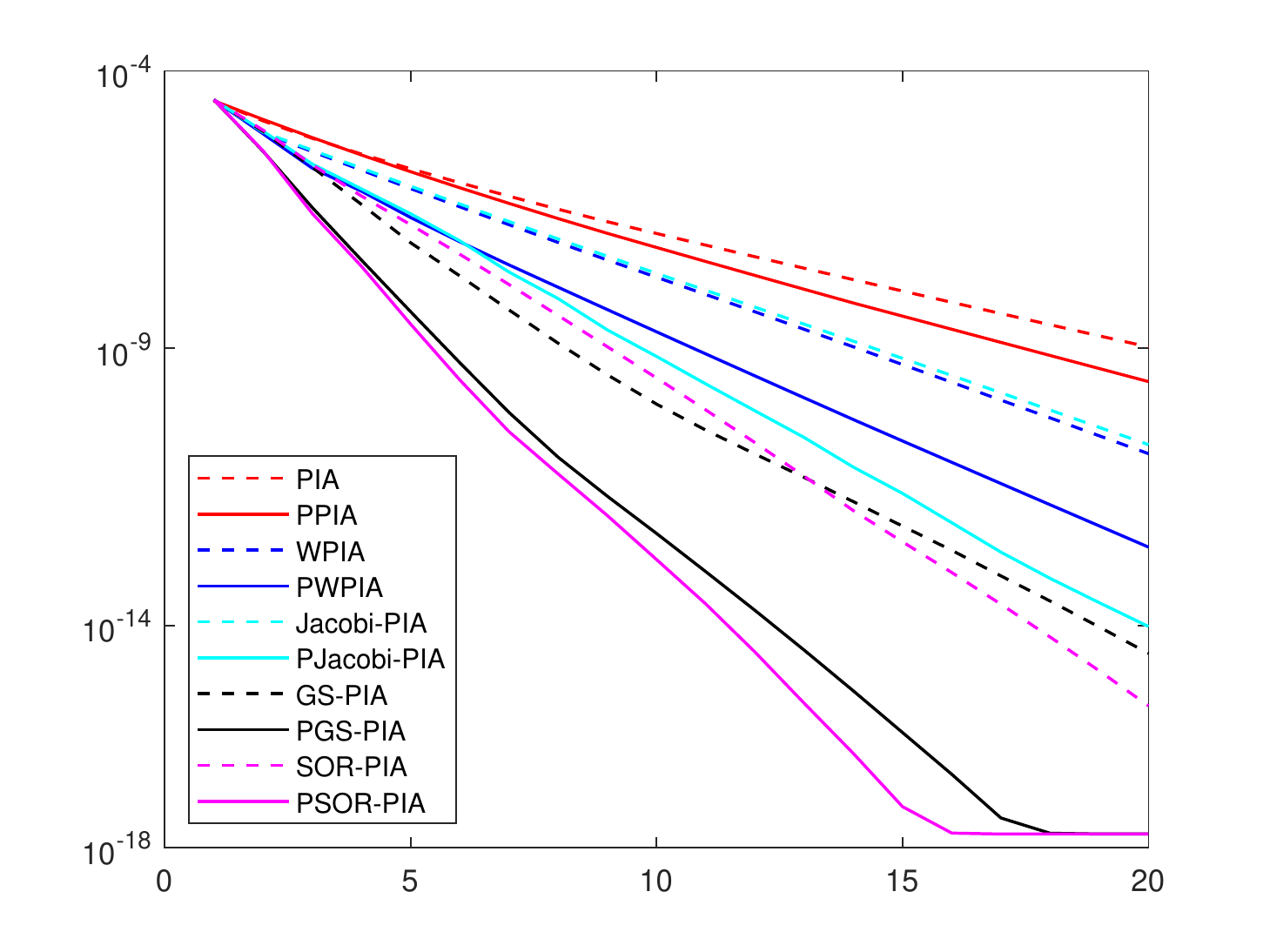}}
\caption{ Comparison of interpolation errors vs number of iterations in Examples \ref{eg1} - \ref{eg5}.}
\label{fig:abcd}
\end{figure}

Firstly, we list in Table \ref{tab1} the spectral radii of the GIMs and their preconditioning techniques for Examples \ref{eg1} - \ref{eg6}. We can observe that the spectral radii of preconditioned GIMs are smaller than those without preconditioning. Therefore, we can expect that our proposed preconditioning GIMs outperform their corresponding GIMs without preconditioning.

By employing the GIMs and the preconditioned GIMs to interpolate the data given in Examples \ref{eg1} - \ref{eg5}, we compare in Figure \ref{fig:abcd} the interpolation errors. In Figure \ref{fig:abcd}, the red, blue, cyan, black, and magenta dashed lines are the curves of interpolation errors obtained by PIA, WPIA, Jacobi--PIA, GS--PIA, and SOR--PIA, respectively. Again, the red, blue, cyan, black, and magenta solid lines are the curves of interpolation errors obtained by the PPIA, the PWPIA, the PJacobi--PIA, the PGS--PIA, and the PSOR--PIA, respectively. The results reported in Figure \ref{fig:abcd} show that the interpolation errors decrease gradually as the number of iterations increases, and the interpolation errors obtained by the preconditioned GIMs are less than those without preconditioning.

Secondly, we turn to compare the computing time. Given a user-defined interpolation error, we list in Table \ref{tab2} the required number of iterations and computing time when we employ the (preconditioned) GIMs to interpolate $n\ (n=1000, 2000)$ points in Example \ref{eg6}. From Table \ref{tab2}, we can see that under the requirement of the same approximation error, the preconditioned GIMs need fewer iterations than those without preconditioning. And the computing times implemented by the PPIA, the PJacobi--PIA, and the PGS--PIA are less than those without preconditioning. This is because the acceleration of the preconditioned technique leads to a reduction in overall costs. The results provide evidence of the fact that the proposed preconditioning technique is efficient in accelerating the convergence of GIMs. It should be noted that some exceptions can be found in Table \ref{tab2}. That is, the computing times implemented by the PWPIA and the PSOR--PIA are more than those without preconditioning. This can be interpreted that there exists a relaxation factor $\omega$, and one has to compute the maximum and minimum eigenvalues to determine the optimal relaxation factor. It is known that this computation is costly, especially on large-scale problems. Thus, the PWPIA and the PSOR--PIA may be slow in time.

\begin{table}[H]
	\centering
	\caption{Required number of iterations and computer time for fixed interpolation error $n$ in Example \ref{eg6}.}
	\label{tab2}
	\begin{tabular}{*{12}{c}}
		\toprule
		\multirow{2}*{$n$}  & \multirow{2}*{$\varepsilon$}  & \multicolumn{2}{c}{PIA} & \multicolumn{2}{c}{WPIA} & \multicolumn{2}{c}{Jacobi-PIA} & \multicolumn{2}{c}{GS-PIA}& \multicolumn{2}{c}{SOR-PIA}\\
		\cmidrule(lr){3-4}\cmidrule(lr){5-6}\cmidrule(lr){7-8}\cmidrule(lr){9-10}\cmidrule(lr){11-12}
		 &    & $k$ & $T$ & $k$ & $T$ & $k$ & $T$ & $k$ & $T$ & $k$ & $T$\\
        \midrule
        \multirow{2}*{$1000$} &1e-10 &35&6.29e-02&24&2.02e+00&24&8.99e-02&16&5.62e-02&15  &1.08e+00\\
                              &1e-12&46&8.66e-02&31&1.88e+00&31&1.19e-01&20   &6.15e-02&19&1.06e+00\\
        \midrule\midrule
        \multirow{2}*{$n$}  & \multirow{2}*{$\varepsilon$}  & \multicolumn{2}{c}{PPIA} & \multicolumn{2}{c}{PWPIA} & \multicolumn{2}{c}{PJacobi-PIA} & \multicolumn{2}{c}{PGS-PIA}& \multicolumn{2}{c}{PSOR-PIA}\\
		\cmidrule(lr){3-4}\cmidrule(lr){5-6}\cmidrule(lr){7-8}\cmidrule(lr){9-10}\cmidrule(lr){11-12}
		  &  & $k$ & $T$ & $k$ & $T$ & $k$ & $T$ & $k$ & $T$ & $k$ & $T$\\
        \midrule
        \multirow{2}*{$1000$} &1e-10&31&5.64e-02&19&2.30e+00&17& 7.79e-02&11&4.50e-02&10&1.28e+00\\
                              &1e-12&40&8.46e-02&25&2.34e+00&21   &9.95e-02&14&5.40e-02&13&1.30e+00\\
        \midrule\midrule
        \multirow{2}*{$n$}  & \multirow{2}*{$\varepsilon$}  & \multicolumn{2}{c}{PIA} & \multicolumn{2}{c}{WPIA} & \multicolumn{2}{c}{Jacobi-PIA} & \multicolumn{2}{c}{GS-PIA}& \multicolumn{2}{c}{SOR-PIA}\\
		\cmidrule(lr){3-4}\cmidrule(lr){5-6}\cmidrule(lr){7-8}\cmidrule(lr){9-10}\cmidrule(lr){11-12}
		 &    & $k$ & $T$ & $k$ & $T$ & $k$ & $T$ & $k$ & $T$ & $k$ & $T$\\
        \midrule
        \multirow{2}*{$2000$} &1e-10&34&3.56e-01&22&9.06e+00&22& 5.99e-01&15&3.62e-01&14&5.03e+00\\
                              &1e-12&44&4.47e-01&29&9.13e+00&29&  6.18e-01&19&3.61e-01&18&5.04e+00\\
        \midrule\midrule
        \multirow{2}*{$n$}  & \multirow{2}*{$\varepsilon$}  & \multicolumn{2}{c}{PPIA} & \multicolumn{2}{c}{PWPIA} & \multicolumn{2}{c}{PJacobi-PIA} & \multicolumn{2}{c}{PGS-PIA}& \multicolumn{2}{c}{PSOR-PIA}\\
		\cmidrule(lr){3-4}\cmidrule(lr){5-6}\cmidrule(lr){7-8}\cmidrule(lr){9-10}\cmidrule(lr){11-12}
		  &   & $k$ & $T$ & $k$ & $T$ & $k$ & $T$ & $k$ & $T$ & $k$ & $T$\\
        \midrule
        \multirow{2}*{$2000$} &1e-10&29&2.88e-01&18&1.05e+01&16& 3.52e-01&10&2.45e-01&10&5.81e+00\\
                              &1e-12&38&3.63e-01&24&1.06e+01&21&  4.50e-01&13&3.15e-01&12&5.78e+00\\
		\bottomrule
	\end{tabular}
\end{table}
The data points and approximate interpolation B-spline curves are plotted in Figures \ref{fig:ex1} -- \ref{fig:ex5}, where the initial cubic B-spline curves are in sub-figures (a), the cubic B-spline curves iterated by the GIMs after $5$ iterations are in sub-figures (b), and the cubic B-spline curves iterated by the preconditioned GIMs after $5$ iterations are in sub-figures (c). It can be seen from Figures \ref{fig:ex1} -- \ref{fig:ex5} that the preconditioned GIMs perform well in interpolating given data sets.
\begin{figure}[H]
\centering
	 \subfigure[$\boldsymbol{C}^{(0)}\left(t\right)$.]{
		\includegraphics[width=0.31\textwidth]{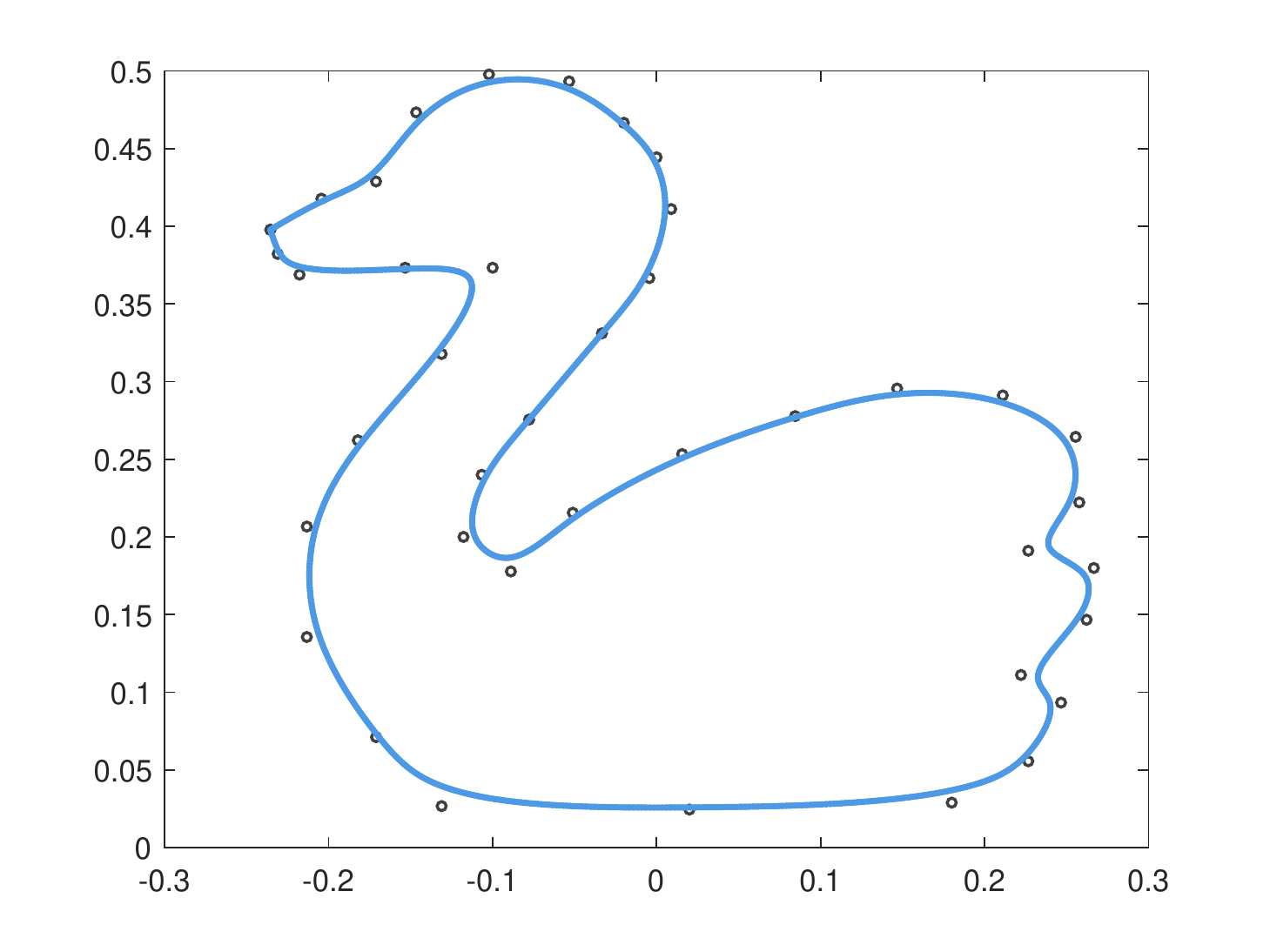}}
     \subfigure[$\boldsymbol{C}^{(5)}\left(t\right)$ by PIA.]{
		\includegraphics[width=0.31\textwidth]{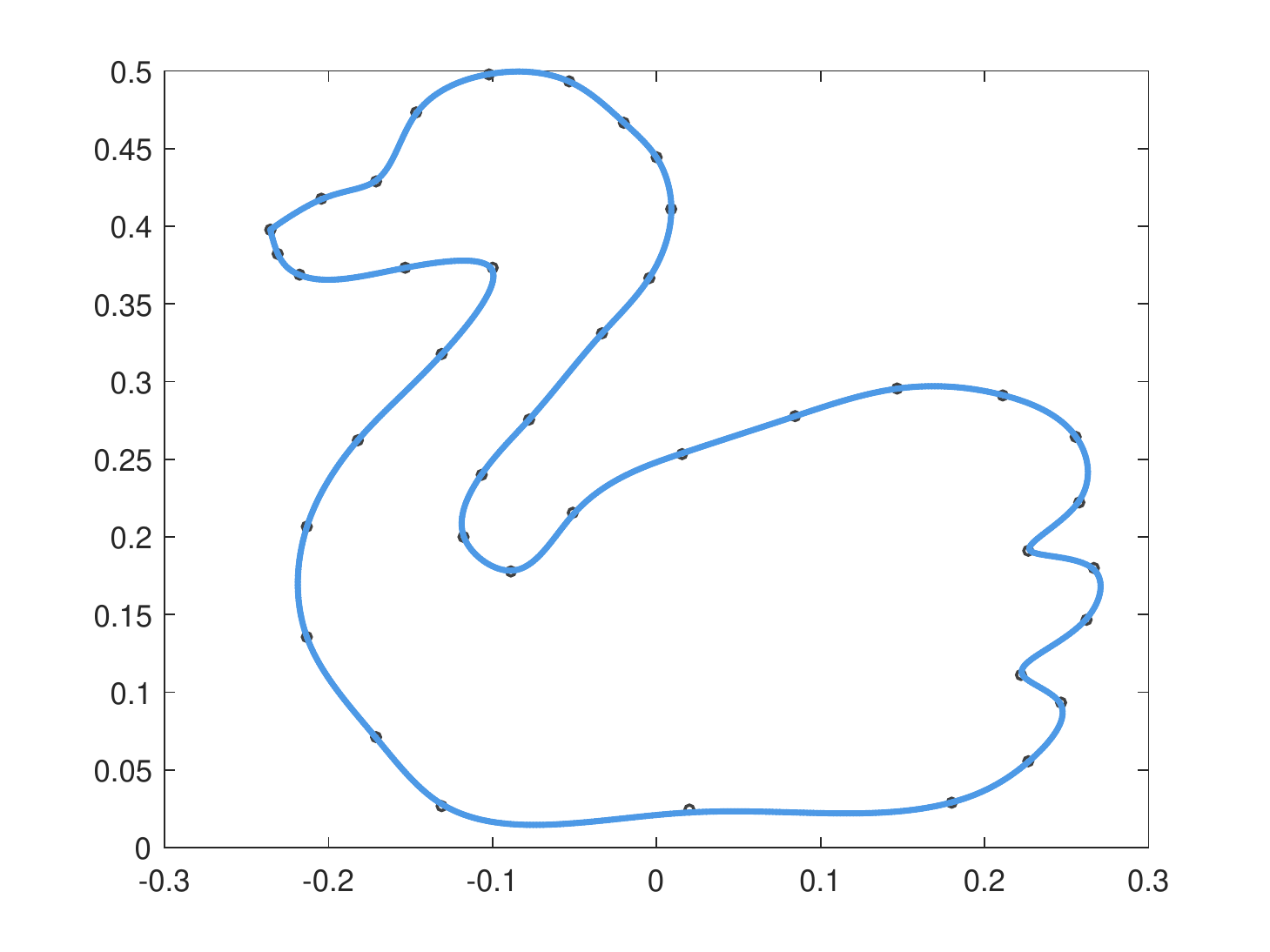}}
     \subfigure[$\boldsymbol{C}^{(5)}\left(t\right)$ by PPIA.]{
		\includegraphics[width=0.31\textwidth]{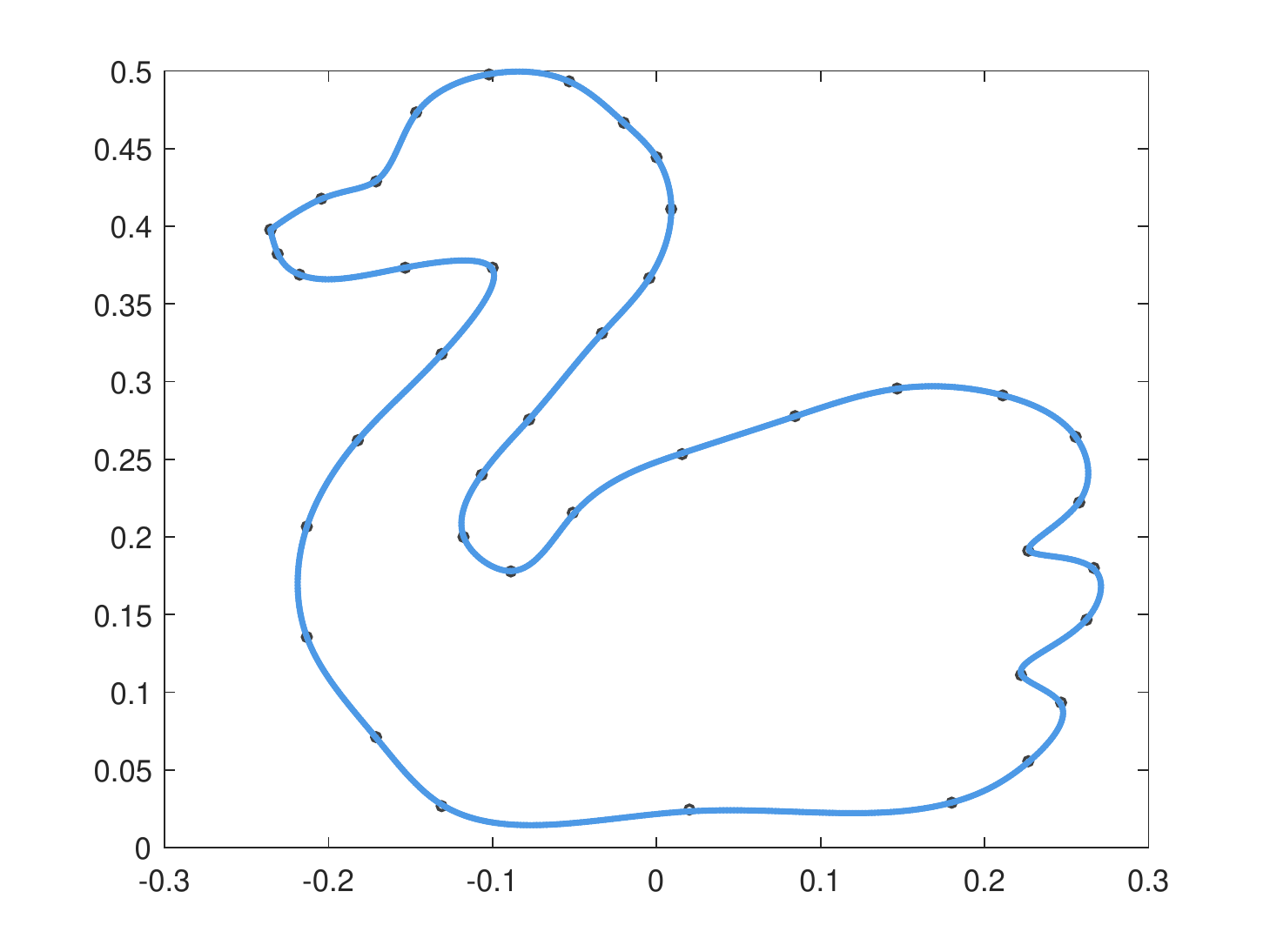}}
\caption{ Cubic B-spline interpolation curves in Example \ref{eg1} obtained by the PIA and the PPIA.}
\label{fig:ex1}
\end{figure}

\begin{figure}[H]
\centering
     \subfigure[$\boldsymbol{C}^{(0)}\left(t\right)$.]{
		\includegraphics[width=0.31\textwidth]{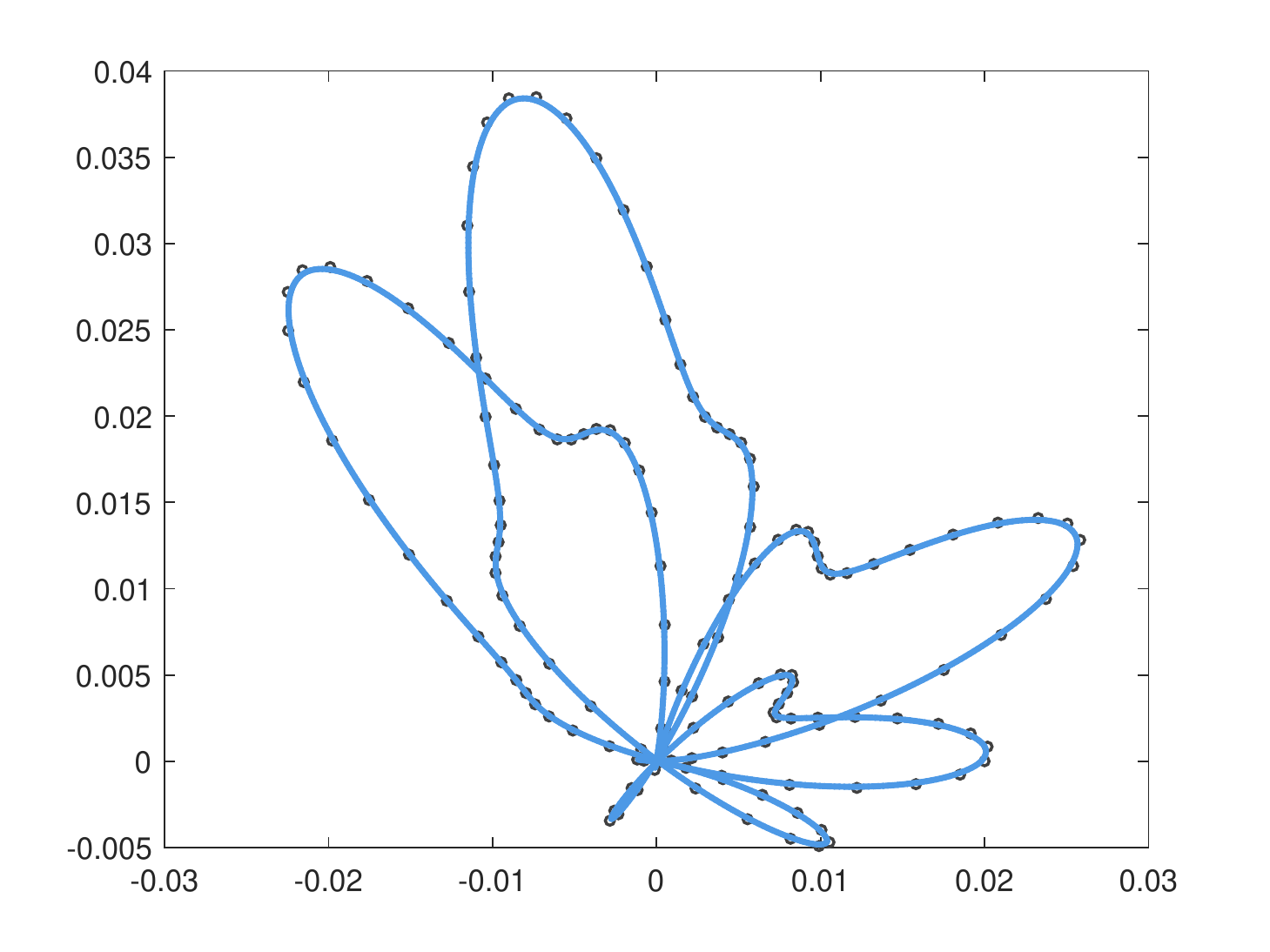}}
     \subfigure[$\boldsymbol{C}^{(5)}\left(t\right)$ by WPIA.]{
		\includegraphics[width=0.31\textwidth]{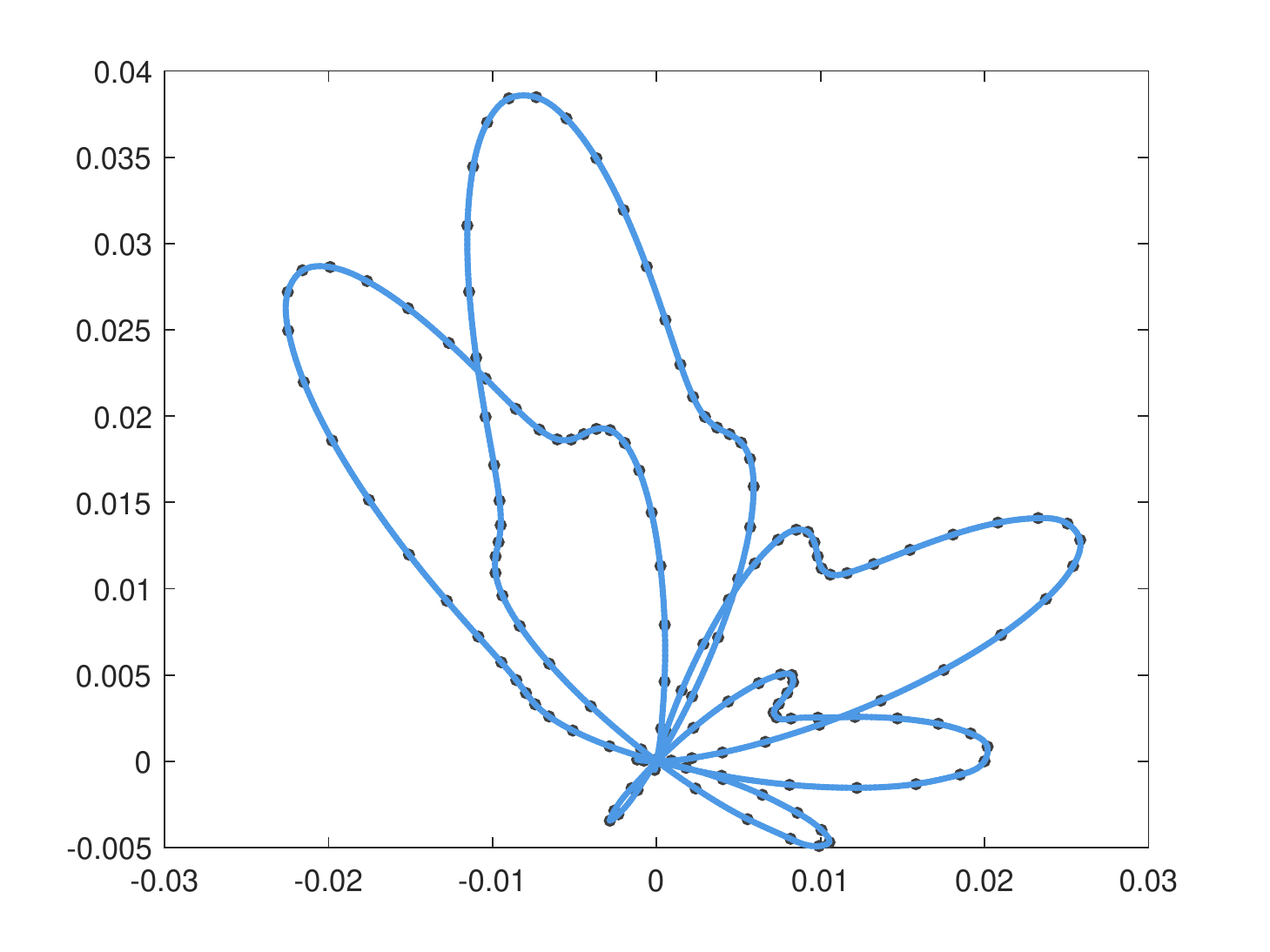}}
     \subfigure[$\boldsymbol{C}^{(5)}\left(t\right)$ by PWPIA.]{
		\includegraphics[width=0.31\textwidth]{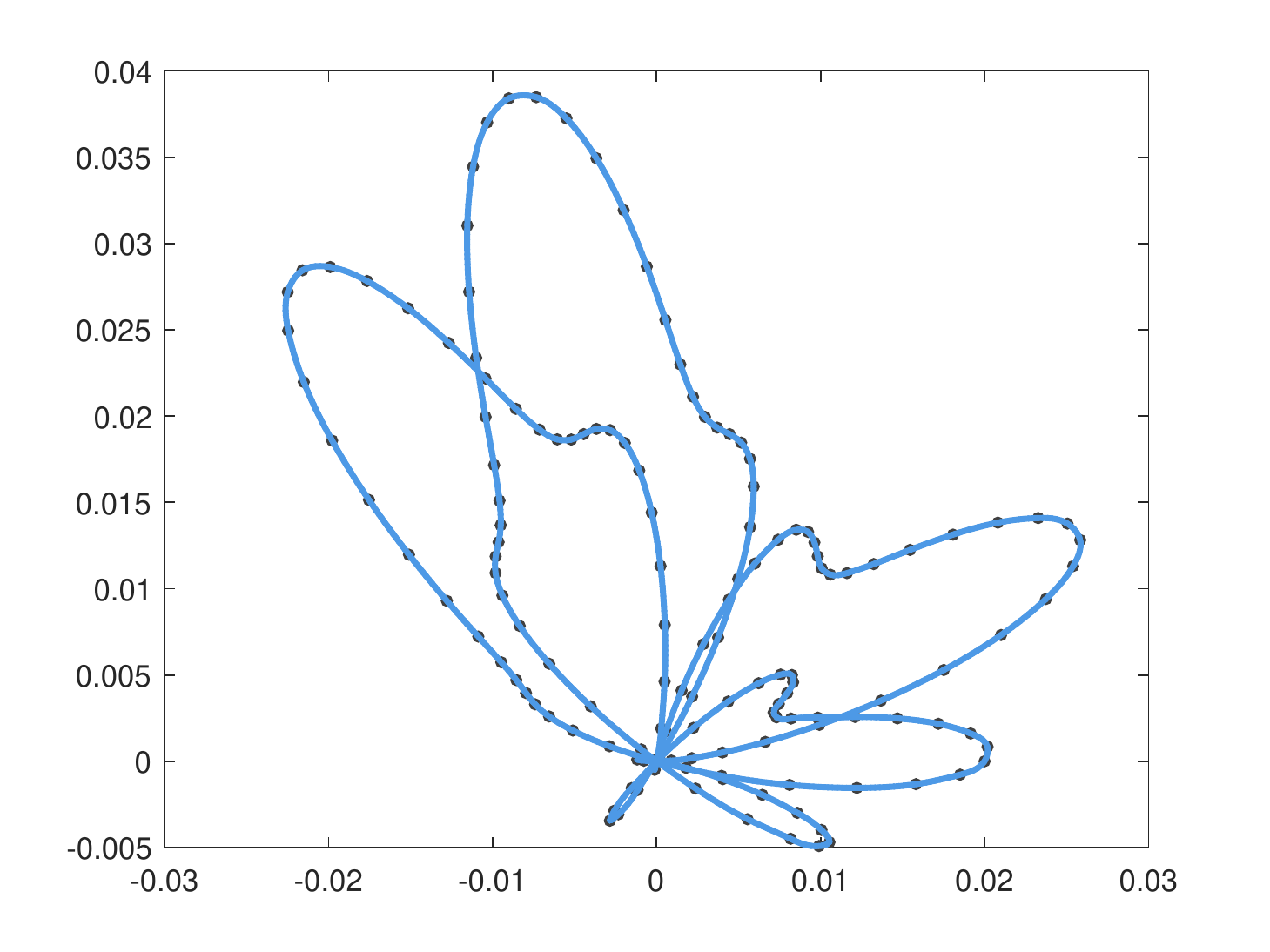}}
\caption{  Cubic B-spline interpolation curves in Example \ref{eg2} obtained by the WPIA and the PWPIA.}
\label{fig:ex2}
\end{figure}

\begin{figure}[H]
\centering
     \subfigure[$\boldsymbol{C}^{(0)}\left(t\right)$.]{
		\includegraphics[width=0.31\textwidth]{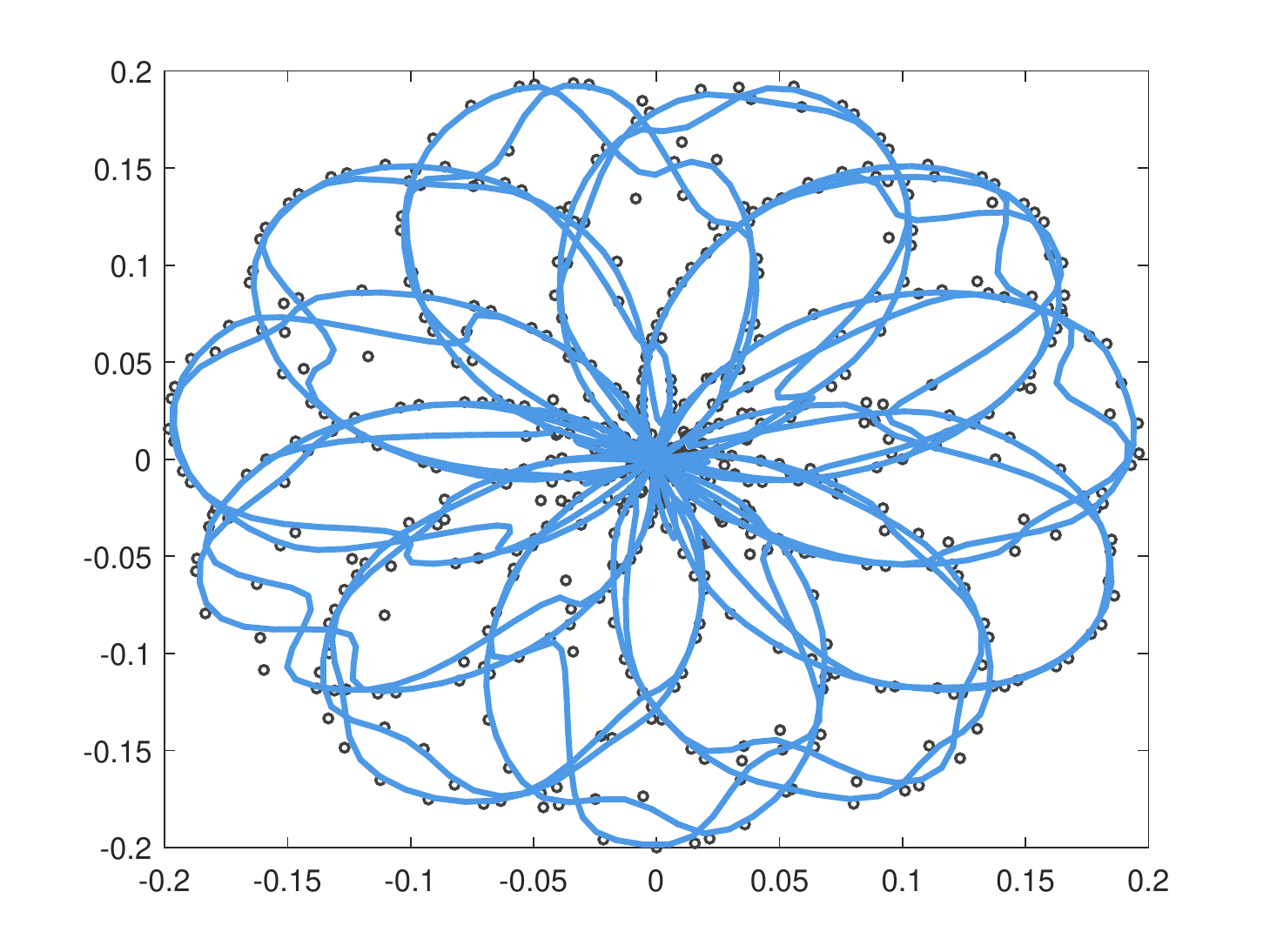}}
     \subfigure[$\boldsymbol{C}^{(5)}\left(t\right)$ by Jacobi-PIA.]{
		\includegraphics[width=0.31\textwidth]{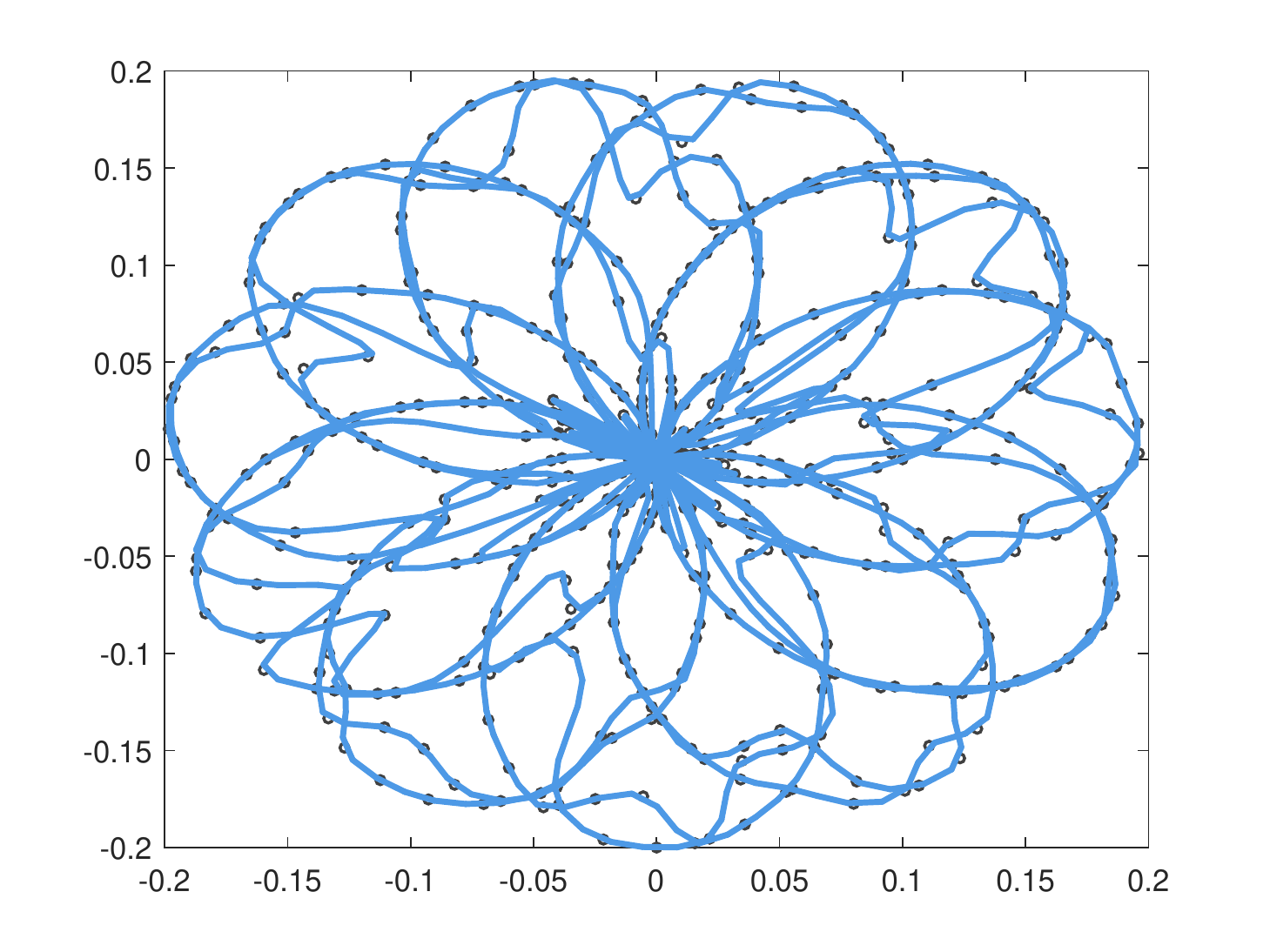}}
     \subfigure[$\boldsymbol{C}^{(0)}\left(t\right)$ by PJacobi-PIA.]{
		\includegraphics[width=0.31\textwidth]{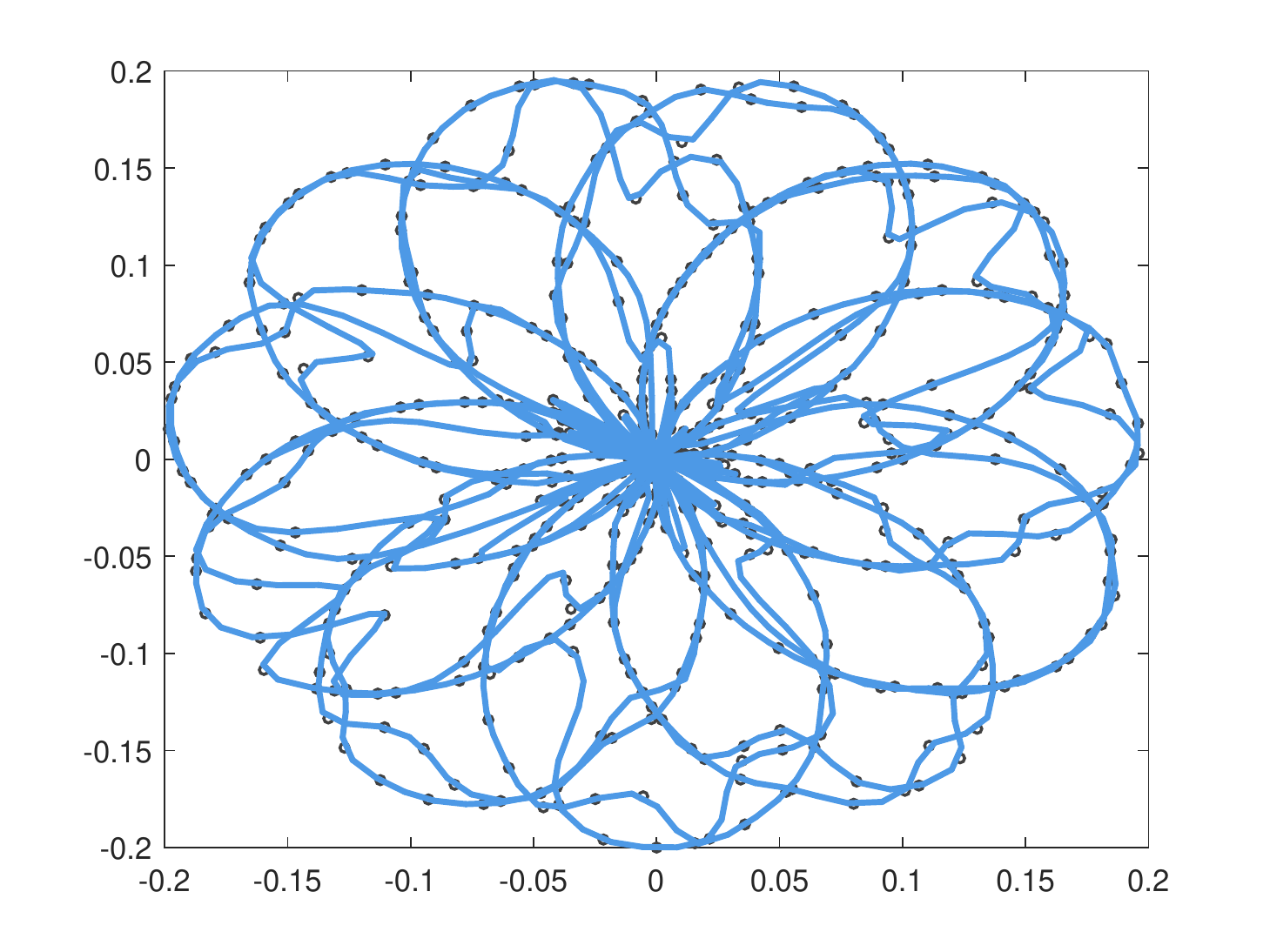}}
\caption{  Cubic B-spline interpolation curves in Example \ref{eg3} obtained by the Jacobi-PIA and the PJacobi-PIA.}
\label{fig:ex3}
\end{figure}

\begin{figure}[H]
\centering
     \subfigure[$\boldsymbol{C}^{(0)}\left(t\right)$.]{
		\includegraphics[width=0.31\textwidth]{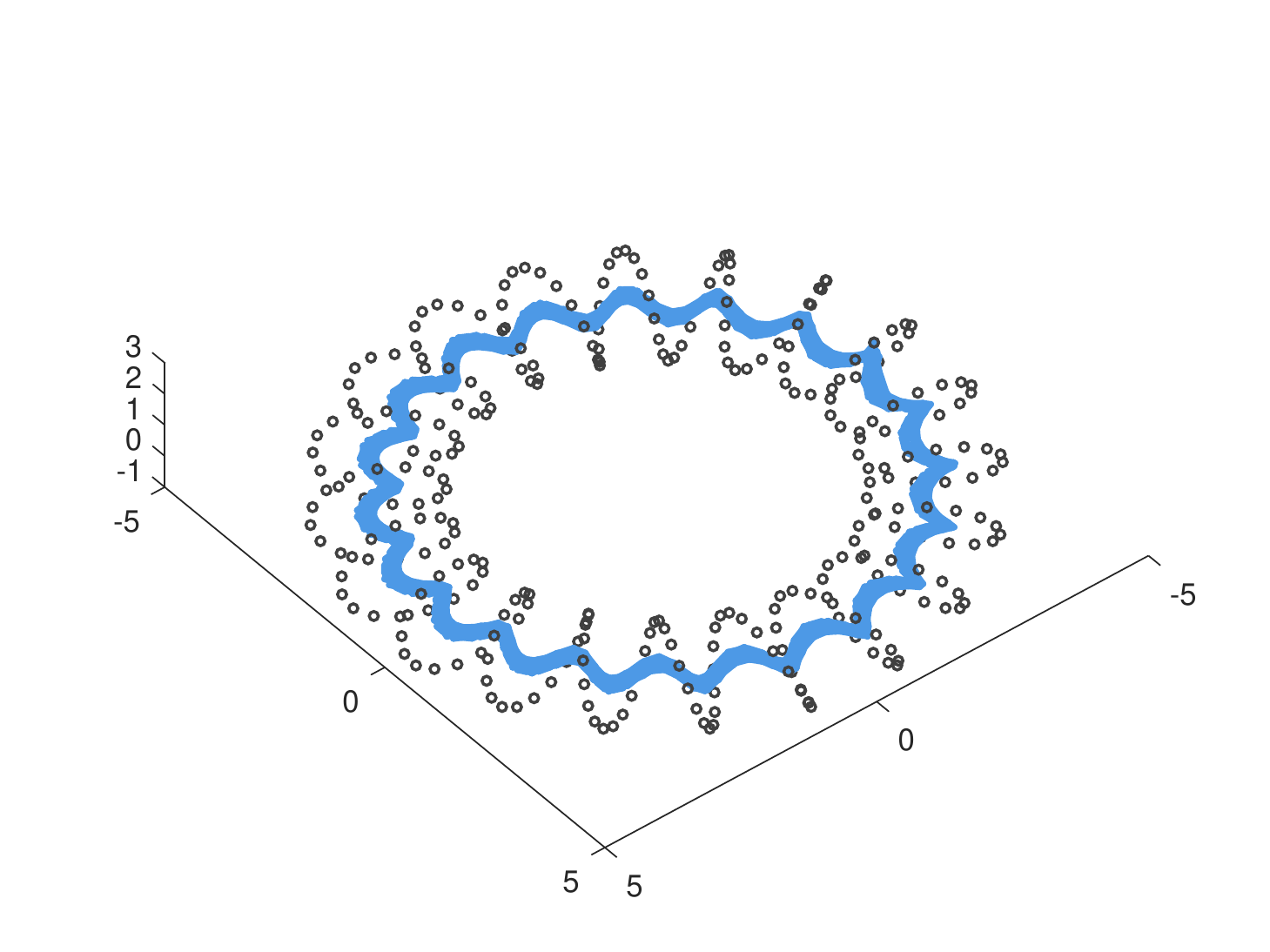}}
     \subfigure[$\boldsymbol{C}^{(5)}\left(t\right)$ by GS-PIA.]{
		\includegraphics[width=0.31\textwidth]{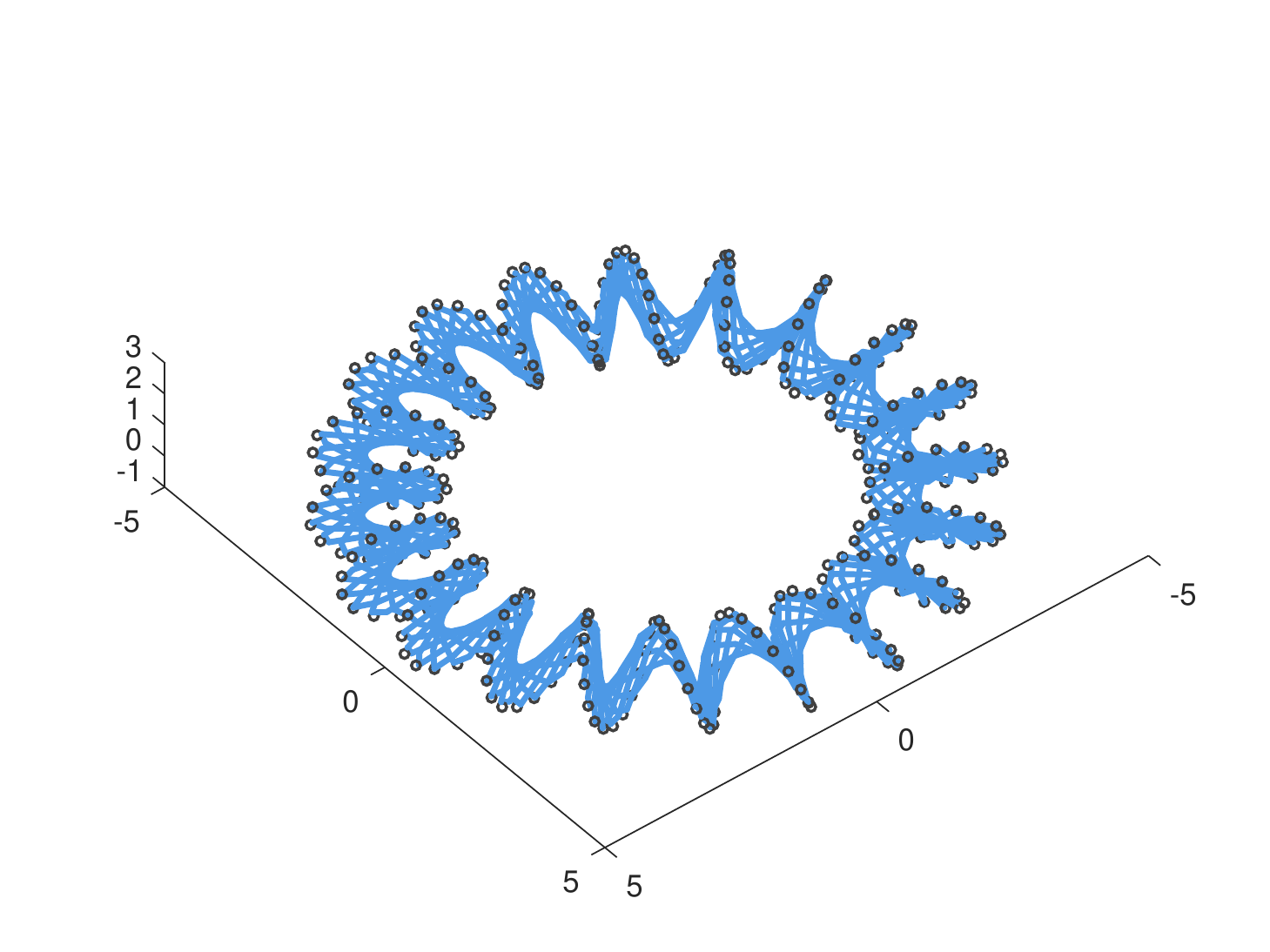}}
     \subfigure[$\boldsymbol{C}^{(5)}\left(t\right)$ by PGS-PIA.]{
		\includegraphics[width=0.31\textwidth]{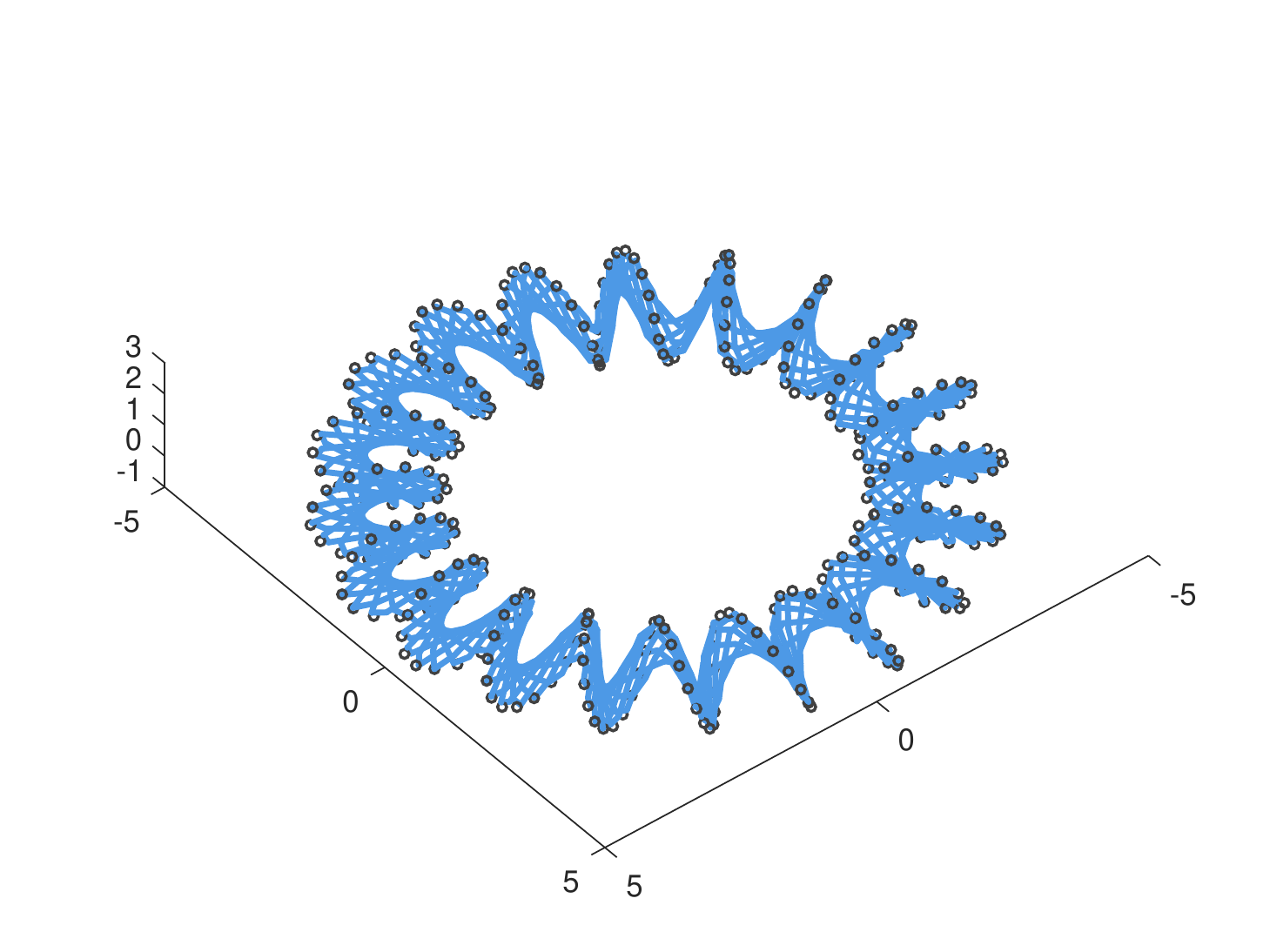}}
\caption{  Cubic B-spline interpolation curves in Example \ref{eg4} obtained by the GS-PIA and the PGS-PIA.}
\label{fig:ex4}
\end{figure}

\begin{figure}[H]
\centering
     \subfigure[$\boldsymbol{C}^{(0)}\left(t\right)$.]{
		\includegraphics[width=0.31\textwidth]{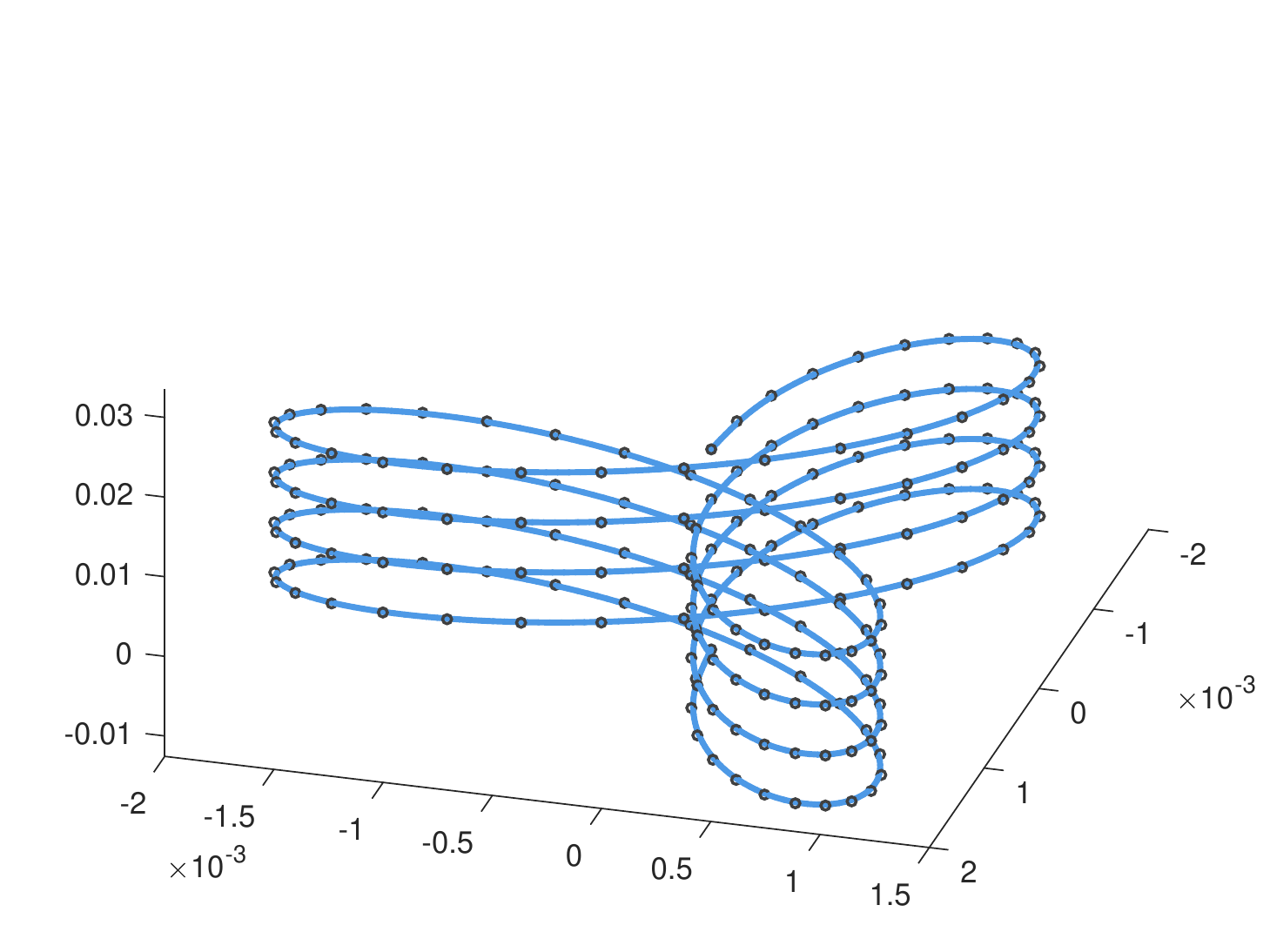}}
     \subfigure[$\boldsymbol{C}^{(5)}\left(t\right)$ by SOR-PIA.]{
		\includegraphics[width=0.31\textwidth]{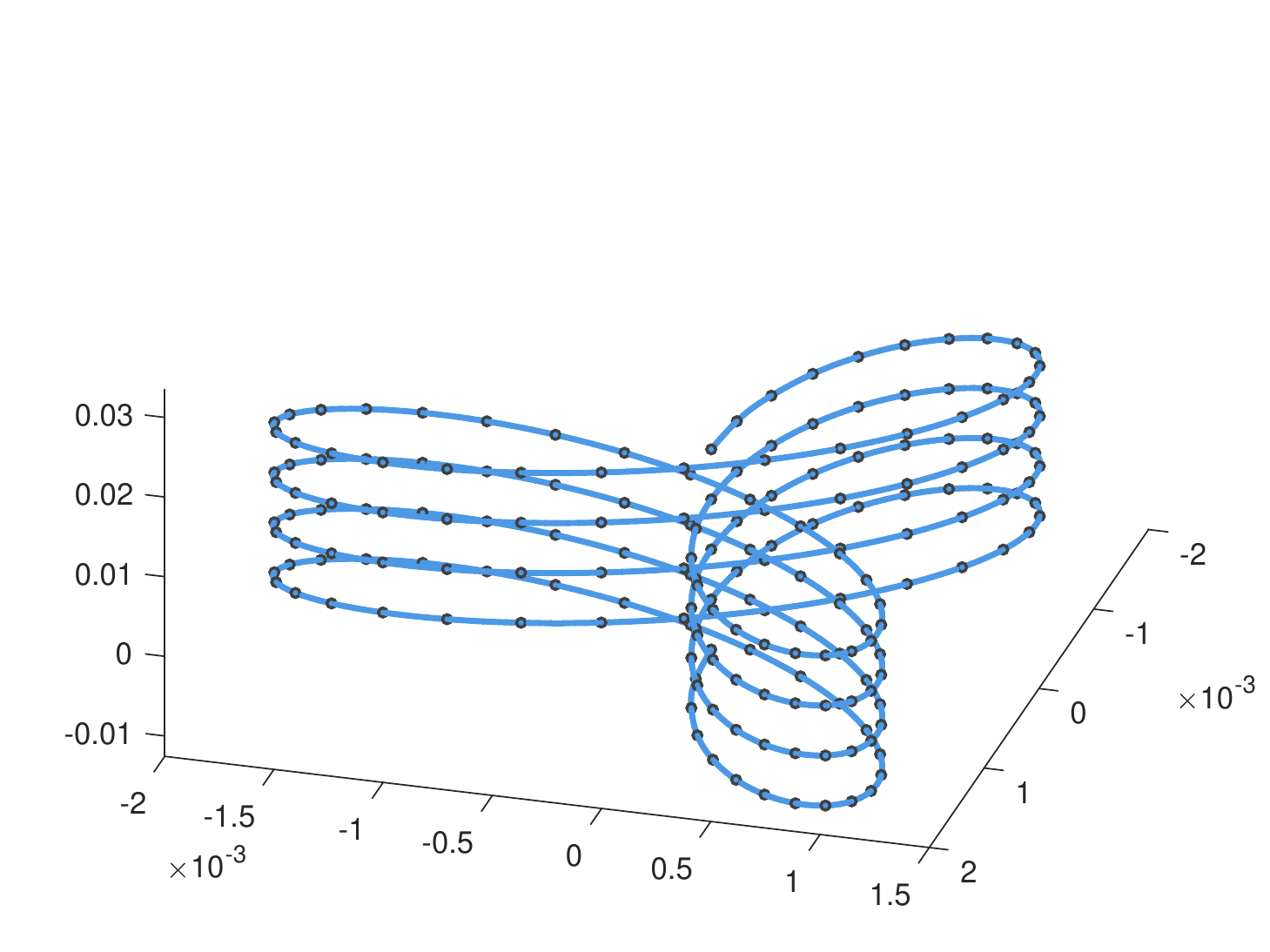}}
     \subfigure[$\boldsymbol{C}^{(5)}\left(t\right)$ by PSOR-PIA.]{
		\includegraphics[width=0.31\textwidth]{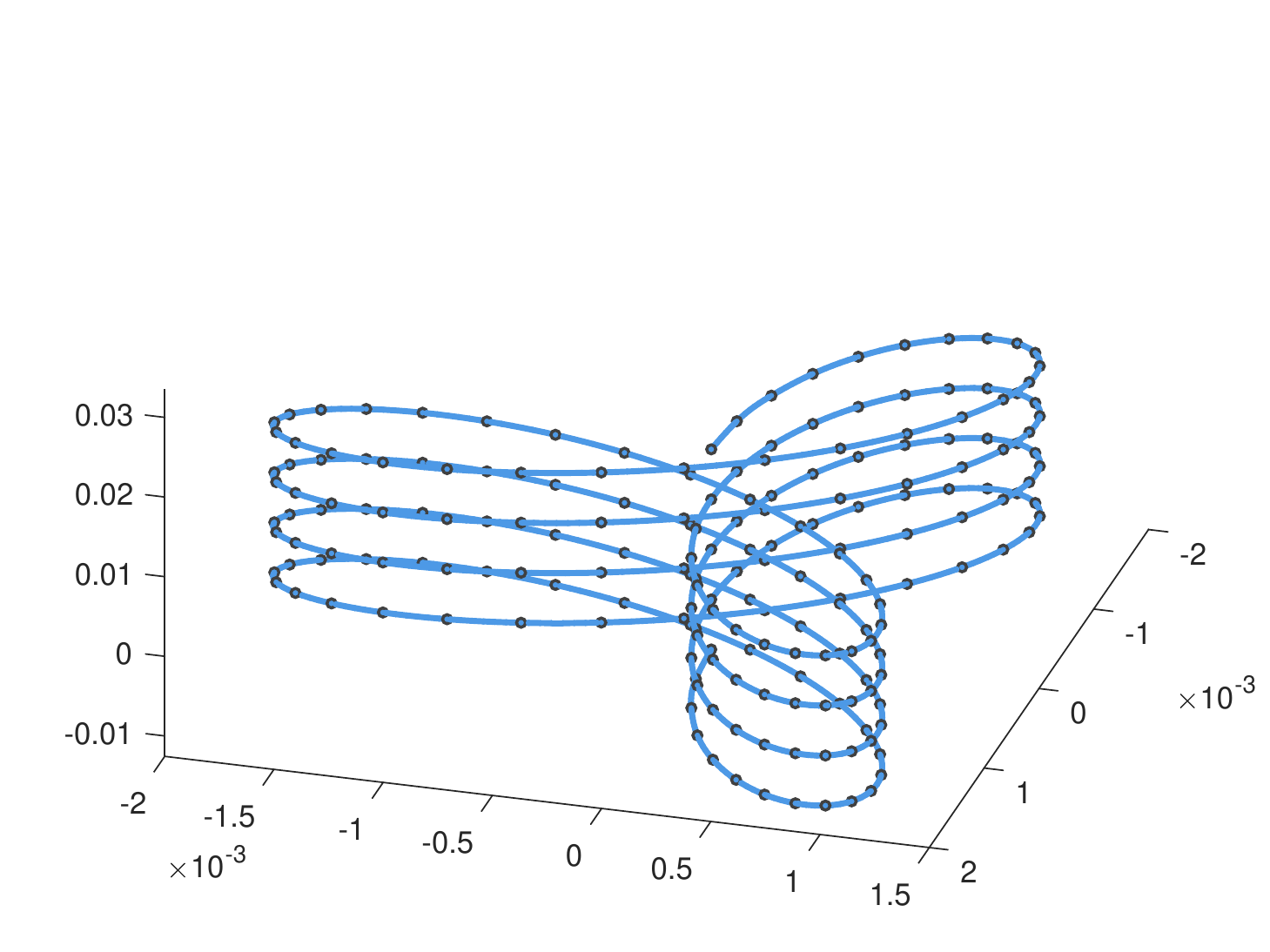}}
\caption{  Cubic B-spline interpolation curves in Example \ref{eg5} obtained by the SOR-PIA and the PSOR-PIA.}
\label{fig:ex5}
\end{figure}

\section{Conclusions}\label{sec5}
In this paper, we have studied the preconditioning technique for the PIA and its variants to interpolate a given set of points. After constructing the preconditioner, we exploited several preconditioned GIMs, which are the preconditioned PIA, the preconditioned WPIA, the preconditioned Jacobi--PIA, the preconditioned GS--PIA, and the preconditioned SOR--PIA. We have shown that the proposed preconditioned GIMs converge. Our numerical experiments demonstrate that the preconditioned GIMs converge and require fewer iterations than those without preconditioning. Moreover, the preconditioning technique is simple and the preprocessing cost is low.
\section*{Acknowledgments}
This study was funded by Natural Science Foundation of China (No. 12101225), Natural Science Foundation of Hunan Province (No. 2021JJ30373), Scientific Research Funds of Hunan Provincial Education Department (No. 21B0790), and Open Research Fund Program of Data Recovery
Key Laboratory of Sichuan Province (No. DRN2104).

\end{document}